\documentclass{amsart}
\usepackage{amsmath,amssymb}
\usepackage{graphicx,subfigure}

\usepackage{color}

\newtheorem{theorem}{Theorem}[section]

\newtheorem{corollary}[theorem]{Corollary}

\newtheorem{lemma}[theorem]{Lemma}

\theoremstyle{definition}

\newtheorem*{definition*}{Definition}

\theoremstyle{remark}
\newtheorem{remark}[theorem]{Remark}

\numberwithin{equation}{section}

\def\hf{\widehat f}

\newcommand{\cc}{\mathrm{c}}

\newcommand{\al}{\alpha}
\newcommand{\be}{\beta}
\newcommand{\de}{\delta}
\newcommand{\ep}{\varepsilon}

\newcommand{\ga}{\gamma}

\newcommand{\la}{\lambda}
\newcommand{\om}{\omega}
\newcommand{\si}{\sigma}
\newcommand{\te}{\theta}
\newcommand{\vp}{\varphi}

\newcommand{\De}{\Delta}
\newcommand{\Ga}{\Gamma}

\newcommand{\Si}{\Sigma}
\newcommand{\Om}{\Omega}

\newcommand{\tu}{\widetilde{u}}

\newcommand{\tX}{\widetilde{X}}
\newcommand{\tv}{\widetilde{v}}
\newcommand{\tw}{\widetilde{w}}

\newcommand{\tOm}{\widetilde{\Om}}
\newcommand{\tPhi}{\widetilde{\Phi}}

\def\hv{\widehat v}

\def\hw{\widehat w}

\def\CC{\mathbb{C}}

\def\RR{\mathbb{R}}

\def\ZZ{\mathbb{Z}}

\def\TT{\mathbb{T}}

\renewcommand\SS{\mathbb{S}}
\newcommand{\cI}{{\mathcal I}}

\newcommand{\cK}{{\mathcal K}}

\newcommand{\cM}{{\mathcal M}}

\newcommand{\cV}{{\mathcal V}}

\newcommand{\pa}{\partial}
\newcommand{\pd}{\partial}
\newcommand\minus\backslash
\newcommand{\id}{{\rm id}}

\newcommand\lan\langle
\newcommand\ran\rangle

\newcommand{\supp}{\operatorname{supp}}

\DeclareMathOperator\Real{Re}

\renewcommand\leq\leqslant
\renewcommand\geq\geqslant
\newlength{\intwidth}

\newcommand\loc{_{\mathrm{loc}}}
\newcommand\LL{_{\mathrm{L}}}

 \DeclareMathOperator\curl{curl}

\newcommand{\Rn}{\mathbb{R}^n}
\newcommand{\RRt}{\mathbb{R}^{n+1}}

\begin{document}

\title[Approximation theorems for parabolic equations]{Approximation theorems for parabolic equations and movement of local hot spots}

\author{Alberto Enciso}
\address{Instituto de Ciencias Matem\'aticas, Consejo Superior de
  Investigaciones Cient\'\i ficas, 28049 Madrid, Spain}
\email{aenciso@icmat.es}

\author{M\textordfeminine \'Angeles Garc\'\i a-Ferrero}
\address{Max Planck Institute for Mathematics, 04103 Leipzig, Germany}
\email{maria.garcia@mis.mpg.de}

\author{Daniel Peralta-Salas}
\address{Instituto de Ciencias Matem\'aticas, Consejo Superior de
 Investigaciones Cient\'\i ficas, 28049 Madrid, Spain}
\email{dperalta@icmat.es}

%
%
\begin{abstract}
  We prove a global approximation theorem for a general parabolic
  operator~$L$, which asserts that if~$v$ satisfies the equation
  $Lv=0$ in a spacetime region~$\Om\subset\RR^{n+1}$ satisfying
  certain necessary topological condition, then it can be approximated
  in a H\"older norm by a global solution~$u$ to the
  equation. If~$\Om$ is compact and~$L$ is the usual heat operator,
  one can instead approximate the local solution~$v$ by the unique
  solution that falls off at infinity to the Cauchy problem with a
  suitably chosen smooth, compactly supported initial datum. These
  results are next applied to prove the existence of global solutions to
  the equation $Lu=0$ with a local hot spot that moves along a
  prescribed curve for all time, up to a uniformly small error. Global
  solutions that exhibit isothermic hypersurfaces of prescribed
  topologies for all times and applications to the heat equation on
  the flat torus are discussed too.
\end{abstract}
\maketitle

\section{Introduction}

The global approximation theory for elliptic equations, which started
with the works of Lax~\cite{Lax} and Malgrange~\cite{Malgrange} in the
1950s, attained its peak with the results of
Browder~\cite{Browder}. Roughly speaking, this theory asserts that
a function satisfying a linear elliptic equation on a closed subset~$\Om$
of~$\RR^n$ can be uniformly approximated by a global solution to the
same equation provided that~$\Om$ satisfies certain necessary
topological condition, namely that its complement in~$\RR^n$ does not
have any bounded connected components. The original motivation for
this theory comes from the classical theorems of Runge and Mergelyan on holomorphic
approximation in complex analysis, which were essentially extended to
harmonic functions and solutions to other elliptic equations with
constant coefficients e.g.\ in~\cite{Gauthier1,Verdera}.

The flexibility granted by global approximation theorems, which
permits to reproduce the behavior of a local solution by a global
solution up to a small error, has recently found applications in
a number of seemingly unrelated contexts such as the nodal sets of eigenfunctions of
Schr\"odinger operators~\cite{JEMS,RMI}, level sets of solutions to the
Laplace and Helmholtz equations~\cite{Adv,Canzani}, the Allen--Cahn
equation~\cite{APDE} and, thanks to an extension of the theory to the
Beltrami field equation (that is, $\curl u-u=0$ on~$\RR^3$), vortex structures in the 3D Euler
and Navier--Stokes equations~\cite{Annals,Acta,NS}. Key to several of
these results has been the development of global approximation
theorems with decay at infinity~\cite{Acta,APDE,JEMS} for the
Helmholtz equation $\De u+u=0$ and the Beltrami field equation.

In contrast, the only approximation theorems available for
parabolic equations are some basic results~\cite{Jones,Diaz1980,Gauthier2} ensuring uniform approximation
on compact sets for the heat equation
\[
-\frac{\pa u}{\pa t}+\De u=0\,.
\]
These results do not provide any control on the growth of the solutions
at infinity (neither in space nor in time) and have not yet found any applications.

Our goal in this paper is to prove flexible global approximation
theorems for general linear parabolic equations and present some
applications to the study of local hot spots and isothermic hypersurfaces. Although the results
and proofs remain valid on a general noncompact manifold with minor modifications, for concreteness we
will restrict our attention to Euclidean spaces and consider parabolic operators of the form
\begin{equation}\label{Lu}
Lu:=-\frac{\pa u}{\pa t}+\sum_{i,j=1}^n a_{ij}(x,t)\frac{\pa^2 u}{\pa x_i\pa x_j}
+\sum_{i=1}^n b_i(x,t)\frac{\pa u}{\pa x_i}+c(x,t)u\,,
\end{equation}
where the space variable~$x$ takes values in $\RR^n$ with~$n\geq2$ and the equation
is uniformly parabolic in the sense that
\[
\inf_{(x,t)\in \RRt} a_{ij}(x,t)\, \xi_i\, \xi_j\geq C_L|\xi|^2
\]
for all $\xi\in\RR^n$, with $C_L$ a uniform constant. The coefficients
of~$L$ and of its formal adjoint, which we denote by $L^*$ are
assumed to be in a suitable parabolic H\"older space $C^{r,\frac r2}(\RR^{n+1})$,
with $r>2$ a real that is not an integer. The definition of these spaces
and the precise regularity assumptions on the coefficients are given in
Section~\ref{s:aboutK}. The parabolic H\"older norm of a function
$f(x,t)$ on a spacetime domain $\Om\subset\RR^{n+1}$ will be denoted by
\[
\|f\|_{r,\Om}:=\|f\|_{C^{r,\frac r2}(\Om)}\,.
\]

\subsection*{Global approximation theorems for parabolic equations}

The first result that we present is a global approximation theorem
on closed sets for general second-order parabolic operators. In order
to state it, if $\Om\subset \RR^{n+1}$ is a set in spacetime, it is
convenient to denote by $\Om(t)$ its intersection
with the time~$t$ slice, that is,
\[
\Om(t):=\{x\in\RR^n: (x,t)\in\Om\}\,.
\]
Furthermore, the complement of a set~$S\subset\RRt$ will be denoted by
$S^\cc:=\RRt\backslash S$ and we will say that a PDE is satisfied in a
closed set $\Om\subset\RR^{n+1}$ if the equation holds in an open
neighborhood of~$\Om$. Let us recall that the topological hypothesis
on the set~$\Om$ that we require in this theorem is known to be
necessary for global approximation already in the case of the heat
equation~\cite{Jones}. One should also observe that, as we are
assuming that~$v$ satisfies a parabolic equation in a neighborhood of
the set~$\Om$, we are in particular making the regularity assumption
that $v$~is of class $C^{r,\frac r2}$ up to the boundary of~$\Om$.

\begin{theorem}\label{T1}
Let $\Om$ be a closed subset in $\RR^{n+1}$ such that
$\Om^\cc(t)$ does not have any bounded connected components for all
 $t\in\RR$. Let us consider a
function $v$ that satisfies the parabolic equation $Lv=0$ in~$\Om$.
Then for any $\de>0$ there exists a solution $u$ of
the equation $Lu=0$ in the whole
space $\RRt$ such that
\[
\|u-v\|_{r+2,\Om}<\de\,.
\]
\end{theorem}

In many applications, a major drawback of all classical global
approximation theorems that is also shared by Theorem~\ref{T1} is that
it does not provide any bounds on the behavior of the global solution
at infinity. Our second theorem ensures that, in the case of the heat
equation, the global
solution~$u$ approximating the local solution~$v$ can be assumed to
decay both at spatial infinity and as $t\to\infty$.

To state our approximation theorem with decay, given a function $f\in
C^\infty_c(\RR^n)$, let us denote by $u:=e^{t\De}f$ the only solution to
the Cauchy problem for the heat equation
\[
\frac{\pd u}{\pd t}-\De u=0\quad\text{in }\RR^{n+1}_+\,,\qquad u(x,0)=f(x)
\]
such that $u(x,t)$ tends to zero as $|x|\to\infty$ for all $t>0$. Here
$\RR^{n+1}_+:=\RR^n\times (0,\infty)$. It
is well known that $u $ can be written in terms of the heat kernel
\[
K(x,t):=c_n t^{-n/2} e^{-|x|^2/4t}
\]
as
\begin{equation}\label{Ku}
u(x,t)=\int_{\RR^n} K(x-y,t)\, f(y)\, dy\,.
\end{equation}

\begin{theorem}\label{T2}
Let $\Om$ be a compact subset of $\RR^{n+1}_+$ such $\Om^\cc(t)$ is connected
for all $t\in\RR$. Given a
function~$v$ satisfying the heat equation $\pd_tv-\De v=0$ on~$\Om$, an integer~$l$ and $\de>0$,
there exists a function $u_0\in C^\infty_c(\RR^n)$ such that
$u:=e^{t\De} u_0$ approximates~$v$ as
\[
\|u-v\|_{C^l(\Om)}<\de\,.
\]
\end{theorem}

In the simpler case of elliptic equations, the existing global
approximation theorems with decay~\cite{Acta,APDE,JEMS} apply to the constant-coefficient
equations $\De u+u=0$ and $\curl u-u=0$ and lead to a pointwise decay
estimate of the form $|u(x)|< C|x|^{-\frac {n-1}2}$, which is very
different from the Gaussian-type decay provided by~\eqref{Ku}. It is worth mentioning that these
approximation results for elliptic equations can be strengthened to deal with equations of nonconstant coefficients via
the integral-type decay conditions described in terms of the Agmon--H\"ormander seminorm. In view
of its relevance for applications, details
are presented in an Appendix (Theorem~\ref{T.appendix}).


\subsection*{Applications to local hot spots and isothermic surfaces}

In the rest of the Introduction we will discuss applications of our
global approximation theorems to the study of the movement of hot
spots and isothermic hypersurfaces for parabolic equations. These
topics have attracted considerable attention for decades, particularly
concerning the large-time behavior of hot
spots~\cite{Chavel,Ishige1,Ishige2,Brasco}, and the existence of
stationary isothermic hypersurfaces and hot spots and the Matzoh ball
soup problem~\cite{MS2,MS1,Sakaguchi}. A related problem, which
concerns the behavior of the second Neumann eigenfunction of a bounded
domain, is the celebrated
hot spots conjecture of Rauch, the first counterexample to which was
found in~\cite{Werner}. Let us recall that a point $X\in\RR^n$ is a {\em (local) hot spot}\/ at time~$t$ of a
solution $u(x,t)$ to the parabolic equation $Lu=0$ if it is a (local)
maximum of~$u(\cdot, t)$, and that an isothermic hypersurface at
time~$t$ is a connected component of a level set of $u(\cdot,t)$. 

Even for the heat equation in the whole Euclidean space~$\RR^n$, despite the large literature on the
subject, not much is known about how local hot spots (or any other
critical points of~$u(\cdot,t)$) can move other than the fact that~\cite{Chavel} if
the initial condition is nonnegative  and compactly supported, the global hot spots must converge
to a point (specifically, the center of mass of the initial
datum). If the initial condition is convex, the convexity is preserved
in the evolution, so there is a unique local (and global) hot spot at
each time. Generalizations of this fact can be found in~\cite{Lions}
and references therein.

We shall next state two theorems showing that the movement of local hots spots
and isothermic hypersurfaces of solutions to a general parabolic
equation on~$\RR^{n+1}$ can be remarkably rich. More precisely, the
first of these results asserts that there are global solutions having
local hot spots that move along any prescribed spatial curve for all
times, up to a small error (allowing for this small error, however, is
key: the construction fails if $\de(t_0)=0$ at some time~$t_0$):

\begin{theorem}\label{T.cp}
Let $\ga:\RR\to\RR^n$ be a continuous
curve in space (possibly periodic or with self-intersections) and take
any positive continuous function on the line $\de(t)$. Then
there is a solution to the parabolic equation $Lu=0$ on
$\RR^{n+1}$ such that, at each time $t\in\RR$,
$u$ has a local hot spot $X_t$ with
\[
|X_t-\ga(t)|<\de(t)\,.
\]
Furthermore, if $L=-\pd_t+\De$ is the heat operator and $[T_1,T_2]\subset[0,\infty)$ is any
finite interval, for any $\de_0>0$ there is a function $u_0\in C^\infty_c(\RR^n)$ such
that, at each time $t\in [T_1,T_2]$, $\tu:=e^{t\De} u_0$ has a local hot spot $\tX_t$ with
\[
|\tX_t-\ga(t)|<\de_0\,.
\]
\end{theorem}

Since there are continuous curves $\ga:\RR\to\RR^n$ whose image is dense
in~$\RR^n$ and one can take an error function $\de(t)$ that tends to
zero as $|t|\to\infty$, an
immediate corollary is the following:

\begin{corollary}\label{C.cp}
There is a solution to the parabolic equation $Lu=0$ on
$\RR^{n+1}$ having a local hot spot~$X_t$ at each time~$t$ such that
$\{X_t: t\in\RR\}$ is dense in~$\RR^n$, that is,
\[
\inf_{t\in\RR}|X_t-x|=0
\]
for all $x\in\RR^n$.
\end{corollary}

The next result, which complements the previous theorem, shows that
there are global solutions that exhibit an isothermic hypersurface of
arbitrarily complicated compact topology (which can change in time). In order
to state this result, let us denote by
\[
\Ga_{u,\al}(t):=\{x\in\RR^n: u(x,t)=\al\}
\]
the level set~$u=\al$ at time~$t$. All the hypersurfaces that we
consider in this paper are of class $C^2$.

\begin{theorem}\label{T.levelsets}
Take a possibly infinite subset $\cI\subset\ZZ$ and, for each index $i\in \cI$, a compact orientable hypersurface without boundary
$\Si_i\subset\RR^n$. If $i$ and $i+1$ belong to~$\cI$, we also assume
that the domains bounded by~$\Si_i$ and~$\Si_{i+1}$ are disjoint. If
$c(x,t)\leq 0$, then there exists a solution of the
equation $Lu=0$ on $\RR^{n+1}$, constants $\al_i>0$ close to zero, and diffeomorphisms
$\Phi_t$ of $\RR^n$ arbitrarily close to the identity in the
$C^{r+2}$~norm such that $\Phi_t(\Si_i)$ is a connected component of
the isothermic hypersurface $\Ga_{u,\al_i}(t)$ whenever $\lfloor
t\rfloor =i$. Furthermore:
\begin{enumerate}
\item If the coefficients of the operator~$L$ depend analytically on~$x$, one can
  take $\al_i=0$ for all~$i\in\cI$.
\item If the set $\cI$ is finite, one can take some $\al>0$ such that $\al_i=\al$
  for all $i\in\cI$.
\item If $L=-\pd_t+\De$ and $\cI$ is a finite subset of
  nonnegative integers,
then one can take~$u:= e^{t\De}u_0$ with $u_0\in C^\infty_c(\RR^n)$.
\end{enumerate}
\end{theorem}

\begin{remark}\label{rem:torus}It is worth mentioning that the behavior described in
Theorems~\ref{T.cp} and~\ref{T.levelsets} is structurally stable,
meaning that a suitably small perturbation of the function~$u(x,t)$
(e.g., small in~$C^1(\RR^{n+1})$)
still presents the same prescribed collection of hot spots or
isothermic hypersurfaces up to a diffeomorphism close to the identity.
\end{remark}

\subsection*{Local hot spots and isothermic surfaces on the flat torus}

It is well known that the large-time behavior of solutions of the heat
equation in a compact Riemannian manifold~$M$ is much more rigid than in
the whole Euclidean space. In particular, for large times one has that
\[
u(x,t)=c_0+c_1e^{-\la_1 t}\psi_1(x) + O(e^{-\la_2 t})\,,
\]
where $0<\la_1\leq \la_2\leq\cdots$ are the eigenvalues of the
Laplacian on the manifold, $\psi_j(x)$ are the corresponding
eigenfunctions and
\[
c_0:=\frac1{|M|}\int_M u(x,0)\, dx\,,\qquad c_1:=\frac{\int_M
  u(x,0)\,\psi_1(x)\, dx}{\int_M
  \psi_1(x)^2\, dx}
\]
are constants. In particular, it is apparent from this formula  that all the local hot spots of~$u(x,t)$
converge
to those of the first nontrivial eigenfunction~$\psi_1(x)$ as $t\to\infty$ for generic
metrics on~$M$ and generic initial data, and this observation can be
easily refined.

Our last objective in this paper is to show that, although global
approximation theorems do not apply in this context, our previous
results do have some bearing on the possible behavior of solutions to
the heat equation on the flat torus $\TT^n$, with
$\TT:=\RR/2\pi\ZZ$. Specifically, we will next state two theorems,
related to Theorems~\ref{T.cp} and~\ref{T.levelsets} above, about
solutions to the heat equation when the space variable takes values on the torus. The
first of them shows that, for suitably short times, there is no obstruction to
the (possibly knotted or self-intersecting) trajectories that local
hot spots can follow on suitably small scales:

\begin{theorem}\label{T.torus}
Let $\ga: [T_1,T_2]\to B_1$ be a (possibly self-intersecting)
continuous curve
contained in the unit $n$-dimensional ball. Given $\de>0$, there
exists an arbitrarily small $\ep>0$ and a solution of the
heat equation ${\pd_t u}-\De u=0$ on $\TT^n\times\RR$
having a local hot spot $X_t$ contained in the ball of
radius~$\sqrt\ep$ for all times in~$[\ep T_1,\ep T_2]$ with the
property that
\[
\bigg|\frac1{\sqrt\ep} X_{\ep \tau}-\ga(\tau)\bigg|<\de
\]
for all $\tau\in [T_1,T_2]$.
\end{theorem}

Intuitively speaking, the reason
for which one considers times of order~$\ep$ and balls of
radius~$O(\sqrt\ep)$ in space is that it is at this scale that the
behavior of solutions to the heat equation on the torus can be shown
to resemble those of the heat equation on the whole Euclidean space,
and viceversa.

Our second theorem in this direction is an analog of
Theorem~\ref{T.levelsets} at these spacetime scales that shows that
there are global solutions of the heat equation on the torus that
feature isothermic hypersurfaces of prescribed, possibly rapidly
changing topologies:

\begin{theorem}\label{T.torus2}
For each index $0\leq i\leq N$, take a compact orientable hypersurface without boundary
$\Si_i$ contained in the unit ball and assume
that the domains bounded by~$\Si_i$ and~$\Si_{i+1}$ are
disjoint. Given $\de>0$ and an integer~$k$,  there
exists an arbitrarily small $\ep>0$ and a
solution of the heat equation ${\pd_t u}-\De u=0$ on
$\TT^n\times\RR$ with the following property: for each time $t\in [0,\ep
N]$ there is a connected component~$S_t$ of the isothermic hypersurface
$\Ga_{u,0}(t)$ contained in the ball of radius~$\sqrt\ep$ and such that
\[
\frac1{\sqrt\ep}S_{\ep \tau}= \Phi_\tau(\Si_i)
\]
for all $\tau\in[0,N]$ with $\lfloor \tau\rfloor =i$, where $\Phi_\tau$
is a diffeomorphism of the unit ball with $\|\Phi_\tau-\id\|_{C^k(B_1)}<\de$.
\end{theorem}

The paper is organized as follows. In Section~\ref{s:aboutK} we
recall a few facts about fundamental solutions that will be needed in
the rest of the paper and set some notations. In
Section~\ref{s:lemmas} we state and prove several technical lemmas
that are key in the proof of the global approximation results. The
proofs of Theorems~\ref{T1} and~\ref{T2} are respectively presented in
Sections~\ref{S.T1} and~\ref{S.T2}, with Section~\ref{S.btu} being
devoted to a refinement of Theorem~\ref{T1} that is crucially employed
in some applications. The results for local hot spots and isothermic
hypersurfaces in~$\RR^{n+1}$ are established in Sections~\ref{S.cp}
and~\ref{S.levelsets}, respectively, while the results for the heat
equation on the torus are presented in Section~\ref{S.torus}. The
paper concludes with an Appendix on global approximation theorems with
decay for elliptic equations.

\section{Fundamental solutions for parabolic equations}\label{s:aboutK}

In this section we state a few results on fundamental solutions for
parabolic equations and establish some notation.

We start by recalling that the parabolic H\"older seminorm of a
function $\vp(x,t)$ defined on a spacetime region
$\Om\subseteq\RR^{n+1}$ is defined as
\[
[\vp]_{\ga,\frac \ga 2,\Om}:=\sup_{\substack{(x,t),(y,s)\in\Om \\ (x,t)\neq(y,s)}}
\frac{|\vp(x,t)-\vp(y,s)|}{\big(|x-y|+|t-s|^{\frac 1 2}\big)^\ga}\,,
\]
where $0<\ga<1$. The parabolic H\"older norm of order~$r$ (and we
will hereafter assume that~$r>0$ is a non-integer real) can then be
written as
\[
\|\vp\|_{C^{r,\frac r 2}(\Om)}:=\max_{2i+|\al|\leq \lfloor r \rfloor }\sup_{(x,t)\in\Om}|D_x^\al\pd_t^i\vp(x,t)|+\max_{2i+|\al|=\lfloor r\rfloor}
[D_x^\al\pd_t^i\vp ]_{r-\lfloor r \rfloor, \frac{r-\lfloor r\rfloor}{2},\Om}\,,
\]
where $\lfloor \cdot \rfloor$ denotes the integral part.

We are now ready to state our regularity assumptions in terms of the
parabolic H\"older norms of the coefficients of the operator~$L$ and
of its formal adjoint, defined as
\begin{equation}\label{eq:L*}
L^*u:=\frac{\pa u}{\pa t}+\sum_{i,j=1}^n a_{ij}(x,t)\frac{\pa^2 u}{\pa x_i\pa x_j}
+\sum_{i=1}^n b^*_i(x,t)\frac{\pa u}{\pa x_i}+c^*(x,t)u
\end{equation}
with
\begin{align*}
b^*_i:=-b_i+2\sum_{j=1}^n\frac{\pa a_{ij}}{\pa x_j}\,,\quad
c^*:=c-\sum_{i=1}^n\frac{\pa b_i}{\pa x_i}+\sum_{i,j=1}^n \frac{\pa^2 a_{ij}}{\pa x_i \pa x_j}.
\end{align*}
Specifically, throughout we will assume that the coefficients of $L$ and $L^*$ are
in a parabolic H\"older space $C^{r,\frac r2}$ with a
non-integer real $r>2$:
\begin{equation}\label{regularity}
a_{ij}, b_i,b_i^*,c,c^*\in C^{r,\frac r2}(\RR^{n+1})\,.
\end{equation}

Under the regularity assumption~\eqref{regularity}, it is standard
(see~\cite[Chapter 1]{Friedman} or~\cite{Escauriaza}) that the operators $L$ and
$L^*$ admit fundamental solutions, which we denote by
$K(x,t;y,s)$ and $K^*(x,t;y,s)$ respectively. Furthermore, the
function $K$ belongs to $C^{r+2,\frac r 2+1}$ outside the diagonal and is bounded as
\begin{equation}\label{eq:Kbounds}
\big|D^\al_xD^\be_y\pd_t^k\pd_s^l K(x,t;y,s)\big|\leq C|t-s|^{-\frac{n+2k+2l+|\al|+|\be|}{2}}e^{-\frac{|x-y|^2}{C|t-s|}}
\end{equation}
for all multiindices with $|\al|+|\be|+2k+2l\leq \lfloor
r\rfloor+2$. It is well known that the connection between $K$ and
$K^*$ is
\begin{equation*}
K^*(x,t;y,s)=K(y,s;x,t)\,,
\end{equation*}
and that the solutions are causal in the sense that
\begin{align*}
K(x,t;y,s)=0\quad &\text{for }t<s\,,\\
K^*(x,t;y,s)=0\quad &\text{for }t>s\,.
\end{align*}

As $K(x,t;y,s)$ is a fundamental solution, given
any function $\vp$ in, say, $C^\infty_c(\RR^{n+1})$, the function
\[
w(x,t):=\int_{\RR^{n+1}} K(x,t;y,s)\, \vp(y,s)\, dy\, ds
\]
satisfies the equation
\[
L w=\vp\,.
\]
In particular,
\[
L \, K(\cdot,\cdot;y,s)=0
\]
for all $(x,t)\neq (y,s)$, and in fact, $LK(x,t;y,s)=
\de(x-y)\,\de(t-s)$ in the sense of distributions.

For future reference, let us record here that the parabolic operators
$L$ and $L^*$ satisfy the unique continuation principle, which one can
state as follows (see~\cite[Section 1.3]{Tataru}):

\begin{theorem}\label{thm:ucp}
Let $L$ be a uniformly parabolic operator as above and consider an open subset in spacetime
$U\subset\RRt$. If $u$ satisfies $Lu=0$ or $L^* u=0$ in a connected subset
$\Om\subset\RR^{n+1}$ containing $U$ and is identically zero in $U$,
it then follows that $u$ must vanish identically in the horizontal
component of $U$ in $\Om$, that is, $u(x,t)=0$ for all $(x,t)\in \Om$ such that $U(t)\neq\emptyset$.
\end{theorem}



\section{Some technical lemmas on the sweeping of poles and discretization} \label{s:lemmas}

In this section we state and prove some key technical lemmas that we
will need for the proof of Theorem~\ref{T1}. Many of the ideas were first introduced by Browder~\cite{Browder}. The first lemma concerns
the behavior of the fundamental solution $K(x,t;y,s)$ when one
perturbs a little the second pair of variables, $(y,s)$:

\begin{lemma}\label{lem:movepole}
Let $U$ be an open subset of $\RRt$ and let $(y,s)$ be a point in
$U$. Then for any $\ep>0$ there is an open
neighborhood~$B_{(y,s)}$ of $(y,s)$ in $U$ such that \begin{equation}\label{eq:diffK}
\|K(\cdot,\cdot;y,s)-K(\cdot,\cdot;y',s')\|_{r+2,U^\cc}<\ep
\end{equation}
for all $(y',s')\in B_{(y,s)}$.
\end{lemma}

\begin{proof}
We can assume that $U$ is bounded without loss of generality. Let us take a proper open subset $\tilde U\subset U$ containing $(y,s)$.
By the estimate \eqref{eq:Kbounds} and the boundedness of $\tilde U$, for any $\de>0$ we can take a compact subset ~$\cK_\de\supset \tilde U$ of ${\RRt}$  such that
\begin{equation*}
\sup_{(y',s')\in \tilde U}\sup_{(x,t)\in \cK_\de^\cc}\big| K(x,t;y',s')\big|<\frac{\de}{2}\,.
\end{equation*}
On the other hand, since
$K(x,t;y',s')$ depends continuously on $(y',s')\in\RRt
\backslash\{(x,t)\}$ and $\cK_\de$ is compact, for each~$(y,s)\in
\tilde U$ there is a small neighborhood $B_{(y,s)}\subset  \tilde U$ of $(y,s)$ such that
\begin{equation*}
\sup_{(y',s')\in B_{(y,s)}}\sup_{(x,t)\in \cK_\de\backslash \tilde U}\big| K(x,t;y,s)-K(x,t;y',s')\big|<\de\,.
\end{equation*}
By the definition of the set $\cK_\de$,
\begin{equation}\label{unif1}
\|K(\cdot,\cdot;y,s)-K(\cdot,\cdot;y',s')\|_{0,  \tilde U^\cc}<\de
\end{equation}
for all $(y',s')\in B_{(y,s)}$.

It is clear that $K(x,t;y,s)- K(x,t;y',s')$, as a function of~$(x,t)$,
solves the parabolic PDE
\[
L\big(K(x,t;y,s)- K(x,t;y',s') \big)=0
\]
in $\tilde U^\cc$. As the closure of the bounded set~$\tilde U$ is
contained in~$U$ and the coefficients of~$L$ are
in~$C^{r,\frac r2}(\RR^{n+1})$, standard interior Schauder estimates
(applied to a uniform cover of~$\tilde U^\cc$ by bounded domains) then allow us
to promote the uniform bound~\eqref{unif1} to
\[
\| K(\cdot,\cdot;y,s)- K(\cdot,\cdot;y',s')\|_{r+2,U^\cc}\leq C
\big\|K(\cdot,\cdot;y,s)- K(\cdot,\cdot;y',s')\big\|_{0,  \tilde U^\cc}< C\de.
\]
The lemma then follows.
\end{proof}

The second lemma shows that, outside a spacetime domain~$U$ satisfying
a certain topological condition, one can approximate the fundamental
solution $K(\cdot,\cdot;y,s)$, understood as a function of the first
pair of variables with $(y,s)$ fixed, by a linear combination of
functions of the form $K(\cdot,\cdot;y_j,s_j)$, where the ``poles''
$(y_j,s_j)$ lie on a prescribed bounded domain contained in~$U$:

\begin{lemma}\label{lem:sweep}
 Let $U$ be a domain in $\RRt$ such that $U(t)$ is connected for all
 $t\in\RR$. Consider a point $(y,s)\in U$ and a bounded domain $\cK\subset U$ such that $\cK(s)\neq\emptyset$.
Then, for any $\ep >0$ there exists a finite set of points $\{(y_j, s_j)\}_{j=1}^J$ in $\cK$ and real constants $\{b_j\}_{j=1}^J\subset\RR$ such that
\begin{equation} \label{eq:sweep}
\bigg\| K(\cdot,\cdot; y,s)-\sum_{j=1}^Jb_j\, K(\cdot,\cdot; y_j, s_j)\bigg\|_{r+2, U^\cc}<\ep\,.
\end{equation}
\end{lemma}

\begin{proof}
We assume $(y,s)$ does not belong to $\cK$, as otherwise the statement
is trivial. Let us take a proper bounded subdomain $\tilde U\subset U$
containing $(y,s)$ and $\cK$. We can assume that $\tilde U(t)$ is
connected for all~$t$.

Consider the space $\cV$ of  all finite linear combinations of the fundamental solution with poles belonging to $\cK$, i.e.
\[
\cV:=\text{span}_\RR\{K(\cdot,\cdot;z,\tau)\!: (z,\tau)\in \cK \}.
\]
Restricting these functions to the complement of $\tilde U$,
$\cV$ can be regarded as a subspace of the Banach space
$C_0(\tilde U^\cc) $ of bounded continuous functions on
$\tilde U^\cc$ that tend to zero at infinity.

By the Riesz--Markov theorem, the dual of
$C_0(\tilde U^\cc)$ is the space $\cM(\tilde U^\cc)$ of
finite regular signed Borel measures on $\RRt$ supported on $\tilde U^\cc$. Let us
take $\mu\in \cM(\tilde U^\cc)$ orthogonal to all functions in $\cV$, i.e.,
\begin{equation*}
\int v \, d\mu=0 \qquad \text{for all } v\in\cV.
\end{equation*}

Now we define a function $F\in L^1\loc(\RRt)$
as
\begin{equation}
F(x,t):=\int K^*(x, t;z,\tau)\, d\mu(z,\tau)
\end{equation}
where $K^*(x, t; y, s)$ is the fundamental solution of the adjoint equation \eqref{eq:L*}. Therefore, $F$ satisfies the equation
\[
L^*F=\mu
\]
in the sense of distributions,
which implies $L^*F=0$ in $ \tilde U$. In addtion, $F$ is identically
zero in $\cK$ because, for all $(x,t)\in \cK$,
\[
F(x,t)=\int K^*(x, t;z,\tau)\, d\mu(z,\tau)= \int K(z,\tau;x, t)\, d\mu(z,\tau)=0
\]
by the definition of the measure~$\mu$.

Hence, since $\tilde U(t)$ is connected, the Unique Continuation Theorem \ref{thm:ucp} ensures that the function $F$ vanishes on the horizontal components of $\cK$ in $ \tilde U$. Particularly, $F$ vanishes in $ \tilde U(s)$. It then follows that
\begin{equation*}
\int K(x,t;y,s)\, d\mu(x,t)=F(y_,s)=0.
\end{equation*}
This implies that $K(\cdot,\cdot;y,s)$ can be uniformly approximated in $C^0( \tilde U^\cc)$ by elements of the subspace $\cV$ as a consequence of the Hahn-Banach theorem.

Let us consider a function
\[
v:=\sum_{j=1}^Jb_j\, K(\cdot,\cdot; y_j, s_j)
\]
in $\cV$ such that
\[
\|K(\cdot, \cdot;y,s)-v\|_{C^0(\tilde U^\cc)}\leq \de
\]
for sufficiently small $\de$.
Since
$$
L\left(K(\cdot, \cdot;y,s)-v\right)=0
$$
in $\tilde U^\cc$, standard parabolic estimates allow us to promote the
above uniform bound to
\[
\|K(\cdot,\cdot; y,s)-v\|_{r+2,U^\cc}\leq C \| K(\cdot,\cdot;
y,s)-v\|_{0,\tilde U^\cc}\leq C\de\,.
\]
We then conclude the desired result.
\end{proof}

The third lemma shows that the ``convolution'' of the fundamental solution
and a compactly supported function can be approximated by a linear
combination of fundamental solution $K(x,t;y_j,s_j)$ with poles
$(y_j,s_j)$ lying on a certain spacetime region:

\begin{lemma}\label{lem:riemsum}
Let $\vp:\RRt\rightarrow\RR$ be an $L^1$ function of compact support.
For any neighborhood $U$ of $\supp \vp$ and $\de>0$ there exists a finite set of points
$\{(y_j, s_j)\}_{j=1}^J$ in $\supp \vp$ and constants $\{c_j\}_{j=1}^J\subset\RR$ such that
\begin{equation}\label{riemsum}
\bigg\|\int_{\RRt} K(\cdot,\cdot;y,s)\vp(y,s)\, dy \,ds-\sum_{j=1}^J c_j K(\cdot,\cdot;y_j, s_j)\bigg\|_{r+2, U^\cc}<\de
\end{equation}
\end{lemma}

\begin{proof}
Let  $U$ be a bounded open neighborhood of the support of~$\vp$.
Since the support of~$\vp$ is compact, we can take a finite collection of balls $\{B_{(y_j,s_j)}\}_{j=1}^J$ as in Lemma~\ref{lem:movepole}, centered at points
$\{(y_j,s_j)\}_{j=1}^J$ in $\supp \vp$ and with
$\ep<\frac{\de}{\|\vp\|_{L^1}}$, such that
$$
\supp \vp\subset \bigcup_{j=1}^J B_{(y_j,s_j)}\,.
$$
Let  $(\chi_j)_{j=1}^J\subset C^\infty_c(\RRt)$ be a partition of unity subordinated to $\{B_{(y_j,s_j)}\}_{j=1}^J$.
  Defining the function $w$ as
  \begin{equation*}
w(x,t):=\sum_{j=1}^J K(x,t;y_j,s_j)\int \chi_j(y,s)\vp(y,s)\, dy\, ds\,,
\end{equation*}
one can write
\begin{multline*}
\int K(x,t;y,s)\vp(y,s)\, dy\, ds-w(x,t)
\\ =\sum_{j=1}^J\int \bigg(K(x,t;y,s)-K(x,t;y_j,s_j)\bigg) \chi_j(y,s)\vp(y,s)\,dy \,ds
\end{multline*}
for any $(x,t)\in U^\cc$. Therefore, by Lemma~\ref{lem:movepole},
\begin{align*}
\bigg\|\int K(\cdot,\cdot;y,s)&\vp(y,s)\, dy\, ds-w\bigg\|_{r+2,U^\cc}\\&
\leq \sum_{j=1}^J\int \|K(\cdot,\cdot;y,s)-K(\cdot,\cdot;y_j,s_j)\|_{r+2,U^\cc} |\vp(y,s)|\chi_j(y,s)\, dy\, ds\\&
< \frac{\de}{\|\vp\|_{L^1}}\int|\vp(y,s)|\sum_{j=1}^J \chi_j(y,s)\, dy\, ds=\de.
\end{align*}
Estimate \eqref{riemsum} then follows taking
$c_j:=\int\chi_j(y,s)\,\vp(y,s)\, dy\,ds$.
\end{proof}

The next lemma in this section shows how the
fundamental solution $K(x,t;y,s)$ can be approximated in a certain spacetime region by a
function satisfying $Lw=0$ in the whole $\RR^{n+1}$:

\begin{lemma}\label{lem:sweepinf}
Let $D$ be an unbounded domain in $\Rn$. Consider the subset of $\RRt$
$U= D\times (T_0, T_1)$ and  a point $(y,s)\in U$.
Then for any $\de > 0$ there exists a function $w$ which satisfies $Lw=0$ in $\RRt$ such that
\begin{equation*}
\|K(\cdot,\cdot;y,s)-w\|_{r+2,U^\cc}<\de.
\end{equation*}
\end{lemma}

\begin{proof}
As $D$ is unbounded, we can take a parametrized curve $z:[0,\infty)\rightarrow D$ without self-intersections such that
\begin{equation*}
z(0)=y, \qquad \lim_{\ell\rightarrow +\infty}|z(\ell)|=+\infty.
\end{equation*}
For each nonnegative integer $m$ we will denote by $\cK_m\subset D$ a
bounded open neighborhood of the point $z(m)$, and $W_m$ a bounded domain containing $\cK_m$ and $\cK_{m-1}$.

Let us consider the following subsets of $\RRt$
\begin{equation*}
\cK_m^T=\cK_m\times(T_0,T_1),\quad W_m^T=W_m\times(T_0,T_1),
\end{equation*}
and denote by $\cV_m$ the space of finite linear combinations of the fundamental solution with poles in $\cK_m^T$, i.e.
\[
\cV_m:=\text{span}_\RR\{K(x,t;z,\tau):(z,\tau)\in \cK_m^T\}.
\]

As $W_1^T(t)$ is connected for all~$t$ and $K_1^T(s)$ is nonempty, by Lemma \ref{lem:sweep} there exists $v_1\in\cV_1$ such that
\begin{equation*}
\|K(\cdot,\cdot;y,s)-v_1\|_{r+2,(W_1^T)^\cc}<\frac{\de}{2}.
\end{equation*}
Since $v_1$ is a finite linear combination of fundamental solutions
with poles in $\cK_1^T$, using Lemma \ref{lem:sweep} again it can be inductively shown that there are functions $v_m\in\cV_m$ such that
\begin{equation*}
\|v_m-v_{m-1}\|_{r+2,(W_m^{T})^\cc}<\frac{\de}{2^m},
\end{equation*}
for all $m\geq 2$. Since the distance between the point $y$ and $W_m$
tends to infinity as $m\rightarrow \infty$, by a standard argument this implies
that $v_m$ converges $C^{r+2,\frac r2+1}$- uniformly on compact sets
to a function $w$ that solves
$$
Lw=0
$$
in $\RRt$.

To finish the proof, it only remains to observe that
\begin{align*}
\|K(\cdot,\cdot;y,s)-w\|_{r+2,U^\cc}&\leq
                                                 \|K(\cdot,\cdot;y,s)-v_1\|_{r+2,U^\cc}
                                                 +\sum_{m=2}^\infty \|v_m-v_{m-1}\|_{r+2,U^\cc}
\\&<\sum_{m=1}^\infty \frac{\de}{2^m}=\de\,,
\end{align*}
where we have used the estimates derived above and the definition of~$w$.
\end{proof}

Finally, the last lemma in this section is a preliminary global
approximation result that we can readily prove using the previous
results in this section and which will be instrumental in the proof of Theorem~\ref{T1}:

\begin{lemma}\label{lem:xtbdpoles}
Let $v$  be a function that satisfies the parabolic equation $Lv=0$ on a compact set $ U\subset \RRt$ such that
$U^\cc(t)$ is connected for all~$t$.
Then for any $\de>0$  there exists a function $w$ satisfying $Lw=0$ in  $\RRt$  and such that it approximates $v$ as
\begin{equation*}
\|v-w\|_{r+2, U}< \de.
\end{equation*}

\end{lemma}

\begin{proof}
Let $\tilde  U$ be an open neighborhood of $ U$ where $Lv=0$. Take a
smooth function $\chi:\RRt\rightarrow\RR$ equal to 1 in a closed set $
U'\subset\tilde  U$ whose interior contains   $ U$ and identically
zero outside $\tilde  U$. We define a smooth extension $\tv$ of the
function $v$ to $\RRt$ by  $\tv:=\chi v$. Since $\tv$ is compactly
supported in spacetime, it follows from standard uniqueness results that
\begin{equation*}
\tv(x,t):=\int K(x,t;y,s)L\tv(y,s) \,dy \,ds\,.
\end{equation*}

The support of $L\tv$ is contained in  $\tilde  U\backslash  U'$.
Then Lemma~\ref{lem:riemsum} ensures we can approximate $\tv$ outside $\tilde  U\backslash  U'$ by a finite combination of fundamental solutions with poles in $\supp L\tv$. In particular, we choose $\{(y_j, s_j)\}_{j=1}^J \subset\supp L\tv$ and real constants $\{c_j\}_{j=1}^J$ such that
\begin{equation*}
\bigg\|\tv-\sum_{j=1}^J c_j K(\cdot,\cdot;y_j,s_j)\bigg\|_{r+2, U}<\frac{\de}{2}.
\end{equation*}

Since $U^\cc(t)$ is connected for all~$t$, for each index $1\leq j\leq J$ there exists an unbounded
domain $D_j\subset\RR^n$ and a finite interval $(T_{0j},T_{1j})$ such that
$(y_j,s_j)\in D_j\times (T_{0j},T_{1j})$ and $D_j\times (T_{0j},T_{1j})\subset
U^\cc$. We can now apply Lemma \ref{lem:sweepinf} to each point
$(y_j,s_j)$, thereby obtaining a function $w$ that satisfies $Lw=0$
in $\RRt$ and such that
\begin{equation*}
\bigg\|\sum_{j=1}^J c_j K(\cdot,\cdot;y_j,s_j)-w\bigg\|_{r+2, U}<\frac{\de}{2}.
\end{equation*}
One then has
\begin{equation*}
\|w-v\|_{r+2, U}=\|\tv-w'\|_{r+2, U}< \de\,,
\end{equation*}
where we have used that $v=\tv$ in~$U$ and the previous bounds.
\end{proof}

\section{Proof of Theorem~\ref{T1}}
\label{S.T1}

Let $\tOm$ be an open neighborhood of $\Om$ where $Lv=0$. Let us take
a smooth function $\chi:\RRt\rightarrow \RR$ that is equal to 1 in  a
closed set $\Om'\subset\tOm$ whose interior contains $\Om$ and is identically zero outside $\tOm$. We define a smooth extension $\tv$ of $v$ to $\RRt$ by $\tv:=\chi v$.

Let $ U_1\subset U_2\subset\dots\subset \RRt$ be an exhaustion of
$\RRt$ by bounded spacetime domains. For concreteness,
one can take $U_m:=B_m\times (-m,m)$. Removing the first sets if
necessary, there is no loss of
generality in assuming that the intersection $U_1\cap
(\RRt\backslash\Om)$ is nonempty. Let $\{\vp_m: \RRt\rightarrow [0,1]\}_{m=1}^\infty$ be a partition of
unity subordinated to $U_{m+1}\backslash U_{m-1}$, where
we have set $U_0:=\emptyset$. (That is, the
support of $\vp_m$ is contained in $U_{m+1}\backslash U_{m-1}$ and
$\sum_{m=1}^\infty \vp_m(x,t)=1$ for all~$(x,t)\in\RR^{n+1}$.)

Consider the  functions
\begin{equation*}
\tv_m(x,t):=\int_{\RR^{n+1}} K(x,t;y,s)\, \vp_m (y,s) \, L\tv(y,s)\, dy\, ds.
\end{equation*}
By the definition of the fundamental solution, it is clear that
\begin{equation*}
L\bigg(\tv-\sum_{m=1}^{M+1}\tv_m\bigg)=0 \text{ in a neighborhood of } U_M.
\end{equation*}
By Lemma~\ref{lem:xtbdpoles} there exists a function $w_1$ satisfying
$Lw_1=0$ on $\RR^{n+1}$ and such that
\[
\|\tv-\tv_1-\tv_2-w_1\|_{r+2,U_1}<1\,.
\]
For a general $M\geq2$, one can now inductively apply
Lemma~\ref{lem:xtbdpoles} to the function
\[
\tv-\sum_{m=1}^{M+1}\tv_m-\sum_{m=1}^{M-1}w_m
\]
to show that there is a function $w_M$ satisfying $Lw_M=0$ on $\RRt$ and such that
\begin{equation}\label{errorM}
\Bigg\|\tv-\sum_{m=1}^{M+1}\tv_m-\sum_{m=1}^{M} w_m\Bigg\|_{r+2,U_M}<\frac{1}{M}\,.
\end{equation}
Therefore the function $\tv$ can be expressed in terms of the
functions $w_m$ as
\begin{equation}\label{eq:tv2}
\tv=\sum_{m=1}^\infty (\tv_m+w_m),
\end{equation}
the convergence being uniform in $C^{r+2,\frac r 2+1}$ on compact subsets of $\RRt$:


Let us now approximate $\tv_m$ by functions $\tw_m$ satisfying
$L\tw_m=0$ in $\RRt$. Since the support of~$\vp_mL\tv$ is contained in
$$
(\tOm\backslash\Om')\cap(U_{m+1}\backslash
U_{m-1}),
$$
which is in turn a bounded subset whose complement contains $\Om\cup U_{m-1}$,
we can apply Lemma~\ref{lem:riemsum} to
construct linear combinations of the form
\[
\hv_m(x,t)=\sum_{j=1}^{J_m}c_{m,j}\, K(x,t;y_{m,j},s_{m,j})
\]
that approximate $\tv_m$ as
\begin{equation}\label{rollo1}
\|\hv_m-\tv_m\|_{r+2, \, \Om\cup U_{m-1}}\leq \frac{\de}{2^{m+1}}.
\end{equation}
Here $\{(y_{m,j},s_{m,j})\}_{j=1}^{J_m}$ is a finite set of points contained in the support of $\vp_m\,
L\tv$ and $c_{m,j}$ are real constants.

We will now apply Lemma \ref{lem:sweepinf} in order to sweep the poles of $\hv_m$ to infinity.
For each~$m$, let us consider an open set $W_m\subset (U_{m-1} \cup
\Om)^\cc$
having $J_m$ connected components of the form
$D_{m,j}\times(T_{0m,j},T_{1m,j})$, where $D_{m,j}$ is an unbounded
domain in $\Rn$ and
\[
D_{m,j}\times(T_{0m,j},T_{1m,j})\cap \bigcup_{j'=1}^{J_m}\{(y_{m,j'},s_{m,j'})\}=\{(y_{m,j},s_{m,j})\}\,.
\]
Notice that it is possible to choose a set $W_m$ as above because
$(y_{m,j'},s_{m,j'})\in U_{m-1}^\cc$ for all $1\leq j'\leq
J_m$. Lemma~\ref{lem:sweepinf} then ensures the existence of  functions $\hw_m$ satisfying $L\hw_m=0$ in $\RRt$ which approximate $\hv_m$ as
\begin{equation}\label{rollo2}
\|\hw_m-\hv_m\|_{r+2,W_m^\cc}\leq\frac{\de}{2^{m+1}}.
\end{equation}

Our goal now is to define a function $u$ as
\begin{equation}\label{eq:w2}
u:=\sum_{m=1}^\infty (w_m+\hw_m).
\end{equation}
To show that this makes sense, we will next prove that the sum on the
right hand side converges $C^{r+2,\frac r 2+1}$- uniformly on compact
subsets of $\RRt$. In order to this, given any compact spacetime
domain $\cK\subset\RR^{n+1}$ let us take an integer $m_0$ such that
$\cK$ is contained in $U_{m_0-1}$ and the intersection $W_m\cap \cK$
is empty for all $m>m_0$. Then for any $M>m_0$ we have the
following estimate on the compact set $\cK$:
\begin{align*}
\bigg\|\sum_{m=m_0}^Mw_m+\sum_{m=m_0}^{M+1}\hw_m\bigg\|_{r+2,\cK}&\leq
\sum_{m=m_0}^{M+1}\big\|\hw_m-\hv_m\big\|_{r+2,\cK}+\sum_{m=m_0}^{M+1}\big\|\hv_m-\tv_m\big\|_{r+2,\cK}\\
&\qquad \qquad \qquad \qquad +\bigg\|\sum_{m=m_0}^{M+1}\tv_m+\sum_{m=m_0}^M w_m\bigg\|_{r+2,\cK}\\
& < \frac{\de}{2}+
  \frac{\de}{2}+\bigg\|\tv-\sum_{m=1}^{M+1}\tv_m-\sum_{m=1}^M w_m\bigg\|_{r+2,\cK}+\|\tv\|_{r+2,\cK}\\
&\qquad \qquad \qquad \qquad \qquad \qquad+\bigg\|\sum_{m=1}^{m_0-1}(\tv_m+w_m)\bigg\|_{r+2,\cK}\\
&\leq \de+\frac{1}{M}+\|\tv\|_{r+2,\cK}+\bigg\|\sum_{m=1}^{m_0-1}(\tv_m+w_m)\bigg\|_{r+2,\cK}.
\end{align*}
Here we have used Equations~\eqref{eq:tv2}--\eqref{rollo2} to pass to
the second line and Equation~\eqref{errorM} to pass to the third. The uniform convergence of the sum on compact sets then
follows. Notice that, by uniform convergence, we immediately infer that $Lu=0$ on $\RRt$.

Finally, from this one can easily show that $u$ approximates $v$ in $\Om$ in the $C^{r+2,\frac r 2+1}(\Om)$ norm:
\begin{align*}
\|u-v\|_{r+2,\Om}&=\|u-\tv\|_{r+2,\Om}\\
&\leq \sum_{m=1}^\infty \| \hw_m-\tv_m\|_{r+2,\Om}\\
&\leq\sum_{m=1}^\infty\bigg(\|\hw_m-\hv_m\|_{r+2,\Om}+\|\hv_m-\tv_m\|_{r+2,\Om}\bigg)\\
&<2\sum_{m=1}^\infty \frac{\de}{2^{m+1}}=\de.
\end{align*}
The theorem then follows.

\begin{remark}\label{R.necessary}
It is classical~\cite{Jones} that the topological condition that
$\Om^\cc(t)$ has no bounded components for all~$t$ (which, when~$\Om$
is compact, is equivalent to saying that $\Om^\cc(t)$ is connected) is necessary, as shown through the
analysis of the heat equation.
\end{remark}

\section{A better-than-uniform approximation lemma}
\label{S.btu}

In this section we present a technical refinement of Theorem~\ref{T1}
which asserts that, when one has a solution of the equation $Lv=0$ defined
on a locally finite union of compact subsets of~$\RR^{n+1}$
(satisfying certain topological conditions), then the approximation by
a global solution of the equation can be better than
uniform. Specifically, the error in this approximation can be chosen
to decrease as fast as one wishes as one goes to infinity. This
refinement will be key in the proof of our results on the movement of
local hot spots.

Before stating the result, let us recall that $\Om
:=\bigcup_{i=1}^\infty\Om_i\subset\RR^{n+1}$ is the {\em locally
  finite}\/ union of the compact subsets~$\Om_i$ if, given any other compact set $\cK\subset\RR^{n+1}$,
the number of subsets~$\Om_i$ that intersect~$\cK$ is always finite.

\begin{lemma}\label{L.technical}
Let the closed set $\Om:=\bigcup_{i=1}^\infty\Om_i$ be a locally finite union of pairwise
disjoint compact connected subsets~$\Om_i$ of~$\RR^{n+1}$ such that the complements of $\Om_i(t)$ are
connected for all $t\in\RR$, and let the function~$v$ satisfy the equation $L v=0$ on~$\Om$. Then one can find a better-than-uniform approximation of~$v$ by
a global solution of the equation, i.e., for any sequence of positive
constants $\{\de_i\}_{i=1}^\infty$ there is a function~$u$
that satisfies the equation~$L u=0$ on~$\RR^{n+1}$ and such that
\begin{equation*}
\|u-v\|_{r+2,\Om_i}<\de_i
\end{equation*}
for all~$i$.
\end{lemma}

\begin{proof}
Since the set $\Om $ is the locally finite union of $\Om_i$, it is
easy to show that there exists an exhaustion $\emptyset=:\cK_0 \subset \cK_1\subset \cK_2\subset\cdots$ by compact sets of $\RR^{n+1}$ such that:
\begin{enumerate}
\item The union of the interiors of the sets $\cK_j$ is $\RR^{n+1}$.
\item For each $j$, the complements of the sets $\cK_j(t)$ and of $(\Om \cup \cK_j)(t)$ are connected for all $t\in\RR$.
  \item If the set $\cK_j$ meets a component $\Om_i$ of $\Om $, then $\Om_i$ is contained in the interior of $\cK_{j+1}$.
\end{enumerate}
An explicit construction of these sets
would be to take a family of concentric balls $\cK_j:=
B_{R_j}$ with radii $R_{j+1}> R_j+1$ chosen so that the following two
conditions are satisfied:
\begin{enumerate}
\item $\pd B_{R_j}\cap \pd \Om=\emptyset$.
\item $R_{j+1}$ is large enough so that all the components of~$\Om$
  that intersect $B_{R_j}$ are contained in $B_{R_{j+1}}$.
\end{enumerate}
The fact that $\Om$ is locally finite ensures that this can be
done. It also follows from this construction that one can safely assume that $\Om\cap (\cK_{j}\backslash \cK_{j-1})\neq
\emptyset$ for all $j\geq1$.

The proof relies on an induction argument that is conveniently presented in terms of a  sequence of positive numbers $\{\ep_k\}_{k=1}^\infty$ related to the approximation parameters $\{\de_i\}_{i=1}^\infty$ and to the above exhaustion by compact sets $\cK_k$ through the conditions
  \begin{equation}\label{epn}
\ep_k<\frac16\min \{\de_i: \Om_i\cap \cK_{k+1}\neq\emptyset\}\quad\text{and}\quad \sum_{j=k+1}^\infty\ep_j<\ep_k\,,
\end{equation}
which are required to hold for all $k\geq1$. For later convenience, we
set $\ep_0:=0$.

The induction hypothesis is that there are functions
$v_j:\RR^{n+1}\to\RR$ satisfying the equation $Lv_j=0$ in $\RR^{n+1}$
and such that, for any integer $J\geq1$, the following estimates hold:
\begin{subequations}\label{induction}
  \begin{align}
\bigg\|v-\sum_{j=1}^Jv_j\bigg\|_{r+2,\Om \cap (\cK_{J+1}\minus \cK_J)}&<\ep_J\,,\label{ind1}\\
\bigg\|    v-\sum_{j=1}^Jv_j\bigg\|_{r+2, \Om \cap (\cK_J\minus \cK_{J-1})}&<\ep_J+2\ep_{J-1}\,,\label{ind2}\\
    \| v_J\|_{r+2, \cK_{J-1}}&<\ep_J+\ep_{J-1}\,.\label{ind3}
  \end{align}
\end{subequations}

Let us start by noticing that, by Lemma~\ref{lem:xtbdpoles}, there
exists a function $v_1$ satisfying the equation $Lv_1=0$ in
$\RR^{n+1}$ and the estimate
\[
\| v-v_1\|_{r+2, \Om \cap \cK_2}<\ep_1\,.
\]
Since we chose $\cK_0=\emptyset$ and $\ep_0=0$, it is a trivial matter that the induction hypotheses~\eqref{induction} hold for $J=1$. We shall hence assume that the induction hypotheses hold for all $1\leq J\leq J'$ and use this assumption to prove that they also hold for $J=J'+1$.

To this end, let us construct a function $w_{J'}$ on the set $\Om \cup \cK_{J'}$ by setting $w_{J'}:=0$ in the set $\cK_{J'}$ and defining $w_{J'}$ on each component $\Om_i$ of the set $\Om $ as
\[
w_{J'}|_{\Om_i}:=\begin{cases}
  v-\sum\limits_{j=1}^{J'}v_j &\text{if }\Om_i\cap (\cK_{J'+2}\minus\stackrel{\circ}\cK_{J'+1})\neq\emptyset\,,\\
  0 &\text{if }\Om_i\cap (\cK_{J'+2}\minus\stackrel{\circ}\cK_{J'+1})=\emptyset\,.
\end{cases}
\]
Here $\stackrel{\circ}\cK_j$ stands for the interior of the set
$\cK_j$, and of course this definition makes sense thanks to the
property~(iii) of the sets~$\cK_j$. The definition of the exhaustion and the first induction hypothesis~\eqref{ind1} guarantee that the function $w_{J'}$ satisfies the equation $Lw_{J'}=0$ in its domain and that one has the estimate
\begin{equation}\label{fm}
\| w_{J'}\|_{r+2,\cK_{J'}\cup(\Om \cap \cK_{J'+1})} \leq \bigg\|v-\sum_{j=1}^{J'}v_j\bigg\|_{r+2, \Om \cap( \cK_{J'+1}\minus \cK_{J'})}<\ep_{J'}\,.
\end{equation}
A further application of Lemma~\ref{lem:xtbdpoles} allows us to take a function $v_{J'+1}$ which satisfies the equation $Lv_{J'+1}=0$ on $\RR^{n+1}$ and which is close to the above function $w_{J'}$ in the sense that
\begin{equation}\label{gm1}
  \| w_{J'}-v_{J'+1}\|_{r+2, \cK_{J'+2}\cap (\Om \cup \cK_{J'})}<\ep_{J'+1}\,.
\end{equation}

Equation~\eqref{gm1} and the way we have defined the function $w_{J'}$
in the set $\cK_{J'+2}\minus \cK_{J'+1}$ ensure that the first induction
hypothesis~\eqref{ind1} also holds for $J=J'+1$. Moreover, from the
relations~\eqref{ind1}, \eqref{fm} and~\eqref{gm1} one finds that, in
the set $\Om \cap (\cK_{J'+1}\minus \cK_{J'})$,
\begin{align*}
  \bigg\|v-\sum_{j=1}^{J'+1}&v_j\bigg\|_{r+2, \Om \cap (\cK_{J'+1}\minus
  \cK_{J'})}\leq \bigg\|v-\sum_{j=1}^{J'}v_j\bigg\|_{r+2, \Om \cap
            (\cK_{J'+1}\minus \cK_{J'})} \\
&\qquad \qquad \qquad \qquad\qquad \qquad \qquad \qquad+ \|v_{J'+1}\|_{r+2, \Om \cap (\cK_{J'+1}\minus \cK_{J'})}\\
  &<\ep_{J'}+ \| w_{J'}-v_{J'+1}\|_{r+2, \Om \cap (\cK_{J'+1}\minus
    \cK_{J'})}+ \| w_{J'}\|_{r+2, \Om \cap (\cK_{J'+1}\minus \cK_{J'})}\\
&<\ep_{J'+1}+2\ep_{J'}\,.
\end{align*}
This proves the second induction hypothesis~\eqref{ind2} for $J=J'+1$. Furthermore,
\[
\| v_{J'+1}\|_{r+2, \cK_{J'}}\leq \|w_{J'}-v_{J'+1}\|_{r+2, \cK_{J'}}+ \| w_{J'}\|_{r+2, \cK_{J'}}<\ep_{J'+1}+\ep_{J'}
\]
by the relations~\eqref{fm} and~\eqref{gm1}, so the third induction hypothesis~\eqref{ind3} also holds for $J=J'+1$. This completes the induction argument.

The desired global solution $u$ can now be defined as
\[
u:=\sum_{j=1}^\infty v_j\,,
\]
with this sum converging uniformly in $C^{r+2,\frac r2+1}$ by a
standard argument thanks to the definition of the constants $\ep_j$ (see conditions~\eqref{epn}) and the third induction hypothesis~\eqref{ind3}. As the functions $v_j$ verify the equation, it is easily checked that the function $u$ also satisfies the equation $Lu=0$ in $\RR^{n+1}$. In addition to this, from the induction hypotheses~\eqref{induction} it follows that, for any integer $J$, in the set $\Om \cap (\cK_{J+1}\minus \cK_{J})$ we have the estimate
\begin{align*}
 \|u-v&\|_{r+2, \Om \cap (\cK_{J+1}\minus \cK_{J})}  \leq
        \bigg\|v-\sum_{j=1}^{J+1}v_j\bigg\|_{r+2, \Om \cap
        (\cK_{J+1}\minus \cK_{J})}\\
&\qquad \qquad \qquad\qquad+\| v_{J+2}\|_{r+2, \Om \cap (\cK_{J+1}\minus \cK_{J})} +  \bigg\|\sum_{j=J+3}^\infty v_j\bigg\|_{r+2, \Om \cap (\cK_{J+1}\minus \cK_{J})}\\
  &<(\ep_{J+1}+2\ep_{J}) +(\ep_{J+2}+\ep_{J+1}) +\sum_{j=J+3}^\infty(\ep_j+\ep_{j-1})\\
  &<6 \ep_{J}
\end{align*}
where we have employed the second condition in the definition of the
constants~\eqref{epn}. To conclude, we now observe that the first
condition in~\eqref{epn} now ensures
\begin{align*}
 \|u-v\|_{r+2, \Om \cap (\cK_{J+1}\minus \cK_{J})}
<\min\{\de_i: \Om_i \cap \cK_{J+1}\neq\emptyset\}\,.
\end{align*}
Therefore, for any~$i$, if $J$ is such that $\Om_i\subset\cK_{J+1}$
and $\Om_i\cap (\cK_{J+1}\backslash\cK_J)\neq\emptyset$ (which implies that
$\Om_i\cap \cK_{J-1}=\emptyset$ by the properties of the
exhaustion~$\{\cK_j\}_{j=1}^\infty$),
\begin{align*}
 \|u-v\|_{r+2, \Om_i}=\max\{ \|u-v\|_{r+2, \Om_i\cap
  (\cK_{J+1}\backslash\cK_J)}\,,\; \|u-v\|_{r+2, \Om_i\cap (\cK_{J}\backslash\cK_{J-1})}\}<\de_i\,.
\end{align*}
By the way the exhaustion has been chosen, for any~$i$ there is always
some~$J$ with this property, so the better-than-uniform approximation lemma then follows.
\end{proof}

\section{Proof of Theorem \ref{T2}}
\label{S.T2}

Without any loss of generality one can assume that $\pd_tv-\De v=0$ in a bounded
open set~$\tOm$ containing~$\Om$ and which is contained in turn in
$B_{R}\times(0,T)$, where we recall that $B_R$ denotes the ball
in $\RR^n$ of
radius~$R$. Taking a smooth cut-off function~$\chi$ as in the proof of
Theorem~\ref{T1}, we infer that the function $\tv:=\chi v$ is
compactly supported and satisfies the equation
\[
\frac{\pd\tv}{\pd t}-\De\tv= \vp
\]
on~$\RR^{n+1}$. In particular, by standard uniqueness properties for
the heat equation one has
\[
\tv(x,t)=\int_{\RRt} K(x-y,t-s)\, \vp(y,s)\, dy\, ds\,,
\]
where the fundamental solution is
\[
K(x,t):= \begin{cases} c_nt^{-n/2} e^{-|x|^2/4t} &\text{ if } t>0\,,\\
0  &\text{ if } t<0\,,
\end{cases}
\]
with $c_n$ a dimensional constant.

If~$\cM$ is any compact set with nonempty interior that is contained in the
complement of $B_{R}\times\RR$ and whose time projection
\[
\{t\in\RR: (x,t)\in\cM \text{ for some } x\in\RR^n\}
\]
contains that of~$\Om'$, Lemmas~\ref{lem:sweep} and~\ref{lem:riemsum} ensure
that there is a finite sum
\[
v_1(x,t):=\sum_{j=1}^J c_j K(x-y_j,t- s_j)
\]
with $c_j$ real constants and $(y_j,s_j)\in\cM$ such that
\[
\|\tv-v_1\|_{C^l(\Om')}<\de\,,
\]
for any fixed $l$ and~$\de>0$ and some compact set
$\Om\subset\!\subset\Om'\subset\!\subset\tOm$ that does not intersect
the support of~$\vp$. Notice that the decay
properties of the fundamental solution~$K(x,t)$ imply that
\begin{equation}\label{decayv1}
\sup_{x\in B_{R}}|D_x^\al \pd_t ^k v_1(x,t)|\leq C
(1+|t|)^{-\frac{n+|\al|}2-k}
\end{equation}
for any fixed~$\al$ and~$k$.

Suppose for the moment that $n\geq3$. Then $v_1$ satisfies the uniform
$L^1$~bound
\[
\sup_{x\in B_{R}}\int_{-\infty}^\infty |D_x^\al\pd_t^k v_1(x,t)|\, dt<C_{\al,k}
\]
for any~$\al$ and~$k$, so the mapping properties of the Fourier transform ensure the Fourier transform of~$v_1$ with respect to
time,
\[
\hv_1(x,\tau):=\frac1{2\pi}\int_{-\infty}^\infty v_1(x,t)\,
e^{-it\tau}\, dt\,,
\]
depends continuously on time and falls off as
\[
\sup_{x\in B_{R}}|D_x^\al \hv_1(x,\tau)|< \frac {C_N}{1+|\tau|^N}
\]
for any~$N$. Of course, $\hv_1(x,\tau)$ is a smooth function of~$x\in
B_R$ because so is~$v_1(x,t)$.

This implies that the inverse Fourier transform formula
\begin{equation}\label{IFT}
D_x^\al \pd_t^k v_1(x,t)=\int_{-\infty}^\infty (i\tau)^k\, D_x^\al\hv_1(x,\tau)\, e^{it\tau}\, d\tau
\end{equation}
holds pointwise for $(x,t)\in B_{R}\times\RR$. In particular, as
$\pd_tv_1-\De v_1=0$ in that set, it follows that
\[
\De \hv_1(x,\tau)-i\tau \hv_1(x,\tau)=0
\]
for all $(x,\tau)\in B_{R}\times\RR$.

We next expand $\hat v_1(x,\tau)$ in a
basis of spherical harmonics on the unit $(n-1)$-dimensional sphere,
which we denote as
\[
\{Y_{mk}(\om): m\geq 0\,,\; 1\leq k\leq d_m\}
\]
with $\om\in
\SS^{n-1}$ and assume to be normalized so that they are an orthonormal
basis of $L^2(\SS^{n-1})$:
\begin{equation}\label{hv1}
\hv_1(x,\tau)=:\sum_{m=0}^\infty\sum_{k=1}^{d_m} \vp_{mk}(r,\tau)\, Y_{mk}(\om)\,.
\end{equation}
Here $r:=|x|$, $\om:=x/|x|$ and $\De_{S^{n-1}} Y_{mk}=-\mu_m Y_{mk}$, where
\[
\mu_m=: m(m+n-2)
\]
is an eigenvalue of the Laplacian on the unit
$(n-1)$-dimensional sphere of multiplicity
\[
d_m:= \frac{2m+n-2}{m+n-2} \binom{m+n-2}{m}\,.
\]
Notice that the coefficient $\vp_{mk}(r,\tau)$
is precisely
\[
\vp_{mk}(r,\tau)=\int_{\SS^{n-1}} \hv_1(r\om,\tau)\, Y_{mk}(\om)\, d\si(\om)\,,
\]
where $d\si$ is the standard measure on the unit sphere, so
$\vp_{mk}(r,\tau)$ is a $C^\infty$ function of $r\in(0,R)$ for all~$\tau$.

As $\hv_1(\cdot,\tau)$ is in $C^\infty(B_R)$ and depends continuously on~$\tau$, the sum converges in
$C^l(B)$ for any integer~$l$ and any smaller ball~$B\subset B_R$, and the convergence is uniform for
$(r,\tau)$ in compact subsets of $[0,R)\times\RR$. For future
reference, we will fix some ball~$B$ such that $\Om\subset B\times
(0,T)$. In particular, it follows that
\begin{align*}
0&= \De\hv_1(x,\tau)-i\tau\,\hv_1(x,\tau)\\
& = \sum_{m=0}^\infty\sum_{k=1}^{d_m} \bigg[\pd_{rr}\vp_{mk}(r,\tau)+\frac{n-1}r\pd_r
  \vp_{mk}(r,\tau) -\bigg(i\tau +\frac{\mu_m}{r^2}\bigg) \vp_{mk}(r,\tau)\bigg]\, Y_m(\om)
\end{align*}
for all $x\in B_R$. Therefore $\vp_{mk}$ satisfies the ODE
\begin{equation}\label{ODE}
\pd_{rr}\vp_{mk}(r,\tau)+\frac{n-1}r\pd_r
  \vp_{mk}(r,\tau)  -\bigg(i\tau +\frac{\mu_m}{r^2}\bigg)\,\vp_{mk}(r,\tau)=0
\end{equation}
for $0<r<R$ and stays bounded at $r=0$. When $\tau\neq0$, the only solution to this ODE
bounded at $r=0$ is
\[
\vp_{mk}(r,\tau)= A_{mk}(\tau)\, r^{1-\frac n2}\, I_{\nu_m}(\sqrt{i
  \tau} r)\,,
\]
where $A_{mk}(\tau)$ is a complex constant that may depend on~$\tau$,
\[
\nu_m:=m+\frac n2-1\,,
\]
and $I_\nu(z)$ denotes the modified Bessel function, defined for
$\Real z>0$ as
\[
I_\nu(z)=\frac1\pi \int_0^\pi e^{z\cos \te} \cos\nu\te\, d\te
-\frac{\sin\nu \pi}\pi\int_0^\infty e^{-z \cosh t-\nu t}\, dt\,.
\]
It is well known
that, for any complex~$z$, this function satisfies the simple exponential
bound
\begin{equation}\label{boundI}
|I_\nu(z)|<C_\nu e^{C_\nu |z|}\,,
\end{equation}
where the constant~$C_\nu$ depends on~$\nu$. When $\tau=0$ the
solution is a polynomial of degree $m+n-2$.

In view of the good convergence properties of the integral~\eqref{IFT} and of the
series~\eqref{hv1}, it is not hard to see that for any~$l$ and any $\de>0$ one
can choose large enough $M$ and~$\tau_0$ such that
\[
\|v_1-v_2\|_{C^l(B\times(-T,T))}<\de\,,
\]
where
\[
v_2(x,t):= \sum_{m=0}^M\sum_{k=1}^{d_m} \int_{-\tau_0}^{\tau_0} \vp_{mk}(r,\tau)\, Y_{mk}(\om)\, e^{it\tau}\, d\tau\,.
\]
Notice that $v_2$ is automatically real-valued.
By construction,
\[
\frac{\pd v_2}{\pd t}-\De v_2=0
\]
in~$\RR^{n+1}$ and $v_2$ is bounded as
\[
\sup_{t\in\RR}|v_2(x,t)|< C e^{C|x|}
\]
as a consequence of~\eqref{boundI}.

Let us consider the smooth function
\[
f(x):=v_2(x,0)\,.
\]
As $\pd_tv_2-\De v_2=0$ and $v_2(x,0)=f(x)$, the fact that $f$ satisfies the
simple exponential bound $|f(x)|< Ce^{C|x|}$ permits to invoke
standard uniqueness results for the heat equation (see~\cite[Sec.~7.1.b]{John}) to conclude that
\[
v_2(x)=\int_{\RR^n} K(x-y,t)\, f(y)\, dy\,.
\]
As the integral converges uniformly in~$C^l$, one can take a smooth
cutoff function (for example, $\chi_1(\cdot/\ep)$, where $0\leq
\chi_1\leq 1$, $\chi_1=1$ on $B_1$ and $\chi_1=0$ outside $B_2$) to
conclude that
\[
\|v_2-u\|_{C^l(B\times(0,T))}<\de
\]
for any small enough~$\ep$, where
\[
u(x,t):= \int_{\RR^n} K(x-y,t)\, f(y)\,\chi_1(y/\ep)\,  dy\,.
\]
By construction,
\[
\|u-v\|_{C^l(\Om)}<C\de\,,
\]
so the theorem follows when $n\geq3$.

When $n=2$, the argument is essentially the same. The fact
that
\[
|v_1(x,t)|<\frac C{1+|t|}
\]
for all~$x\in B_R$ and $t\in\RR$
immediately implies that its Fourier transform $\hv_1(x,\cdot)$ is in
$L^p(\RR)$ for all $2\leq p<\infty$ and
satisfies
\[
\sup_{x\in B_R} \|\hv_1(x,\cdot)\|_{L^p(\RR)}< C_p\,.
\]
Furthermore, as the decay estimate~\eqref{decayv1} shows that
\begin{equation}\label{refdecay}
\sup_{x\in B_R}\|\pd_t^kD_x^\al v_1(x,\cdot)\|_{L^1(\RR)} <C
\end{equation}
whenever $k+|\al|\geq1$, it follows from the mapping properties of the
Fourier transform that
\[
\widehat{\pd_t^kD_x^\al v_1}(x,\tau)= (i\tau)^k\, D_x^\al \hv_1(x,\tau)
\]
is a continuous function of~$\tau$ and uniformly bounded
in~$x\in B_R$ for~$k+|\al|\geq1$. Therefore, $D_x^\al \hv_1(x,\tau)$
decays rapidly at infinity, that is, for all~$N$ and~$\al$ there exists
a constant $C_{N,\al}$ such that
\[
\sup_{x\in B_R}|D_x^\al \hv_1(x,\tau)|<\frac{C_N}{|\tau|(1+|\tau|^N)}\,,
\]
so $\tau^k D_x^\al\hv_1(x,\tau)$ is an integrable function of~$\tau$ uniformly~in $x\in
B_R$, for any $k\geq0$.
Hence the inverse Fourier transform integral
\[
\pd_t^k D_x^\al  v_1(x,t)=\int_{-\infty}^\infty (i\tau)^k\, D_x^\al\hv_1(x,\tau)\, e^{it\tau}\, d\tau
\]
converges uniformly on~$(x,t)\in B_R\times\RR$ for all $k$ and~$\al$, so the rest of the argument
carries over to the case $n=2$ without any further modifications.

\section{Proof of Theorem \ref{T.cp}}
\label{S.cp}

We start by noticing that the function $v'(x,t):=e^{-\la_0t}v(x,t)$
satisfies the equation
\[
Lv'-\la_0 v'=0
\]
and that the local hot spots of~$v'$ coincide with those
of~$v$. Hence, by taking $\la_0:=\sup_{(x,t)\in \RR^n+1}c(x,t)$
(which is finite by the assumptions on the coefficients of the
operator~$L$) and considering the equation for~$v'$ instead of that
for~$v$, we infer than one can assume $c(x,t)<0$ without any loss
of generality.

We can also assume that the curve $\ga(t)$ is
smooth, since otherwise in the argument below one
can replace $\ga(t)$ by any function $\widetilde\ga(t)$ in $C^\infty(\RR)$ such
that
\[
|\ga(t)-\widetilde\ga(t)|<\frac{\de(t)}4\,.
\]
This is a consequence of the density of smooth functions in the space
of continuous functions $C^0(\RR)$ (see e.g.~\cite[Theorem~2.2]{Hirsch}).

For each integer~$k$, let us consider the bounded spacetime domain
\[
\Om_k:= \bigg\{ (x,t)\in\RR^{n+1}: k-\frac15< t <k+\frac65,\; \Big|
x-\ga(t)+(-1)^k \frac{\de_k} 2 e\Big|<\frac{\de_k} 4\bigg\}\,,
\]
where $e\in\RR^n$ is a fixed unit vector and $\de_k$ is a small
constant that we will define in terms of the function $\de(t)$ later
on. Notice that the closures of
these domains are pairwise disjoint and
\[
|x-\ga(t)|<{\de_k}
\]
for all $(x,t)\in\Om_k$.

Let $\phi$ be a smooth nonnegative function on $\RR^n$, not identically zero, that is supported in
the ball of radius~$\frac14$. In each domain~$\Om_k$ we shall construct a function~$v_k$ such that
\[
Lv_k=0\quad \text{in }\Om_k\,,
\]
$v_k(x,t)=0$ if $x\in \pd\LL\Om_k$ and
\[
v\Big(x,k-\frac15\Big)=\phi\bigg(\frac{x-\ga(k-\frac15)+(-1)^k \frac{\de_k} 2 e}{\de_k}\bigg)\,.
\]
Here
\[
\pd\LL\Om_k :=\bigg\{ (x,t)\in\RR^{n+1}: k-\frac15< t <k+\frac65,\;
\Big|x-\ga(t)+(-1)^k \frac{\de_k} 2 e\Big|=\frac{\de_k} 4\bigg\}
\]
is the lateral boundary of~$\Om_k$.

To this end, let us introduce the new coordinates
\[
\xi:= x-\ga(t)+(-1)^k \frac{\de_k} 2 e\,,\qquad \tau:= t\,.
\]
This change of coordinates maps the domain $\Om_k$ onto $B_{\frac{\de_k}4}\times
(k-\frac15, k+\frac65)$. With some abuse of
notation, we will still denote by $v_k(\xi,\tau)$ the expression of
the function $v_k(x,t)$ in the new coordinates. The equation for $v_k$ then reads
\begin{multline*}
-\frac{\pa v_k}{\pa \tau}+ \sum_{i,j=1}^n a_{ij}\bigg(\xi +\ga(\tau)- (-1)^k \frac{\de_k} 2 e,\tau\bigg)\frac{\pa^2 v_k}{\pa \xi_i\pa \xi_j}
\\+\sum_{i=1}^n \bigg[\dot \ga_i (\tau) +b_i\bigg(\xi +\ga(\tau)- (-1)^k \frac{\de_k} 2
e,\tau\bigg)\bigg]\,\frac{\pa v_k}{\pa \xi_i}+c\bigg(\xi +\ga(\tau)-
(-1)^k
\frac{\de_k} 2 e,\tau\bigg)\, v_k=0\,,
\end{multline*}
in the domain $|\xi|<\frac{\de_k}4$, $k-\frac15< \tau <k+\frac65$, and we
have the initial and boundary conditions
\[
v_k(\xi,\tau)=0\quad \text{if } |\xi|={\de_k}/4\,,\qquad v_k\Big(\xi, k-\frac15\Big)=\phi(\xi/{\de_k})\,.
\]
It is standard that there is a unique solution to this
initial-boundary value problem. As we can assume that $c\leq0$ without
any loss of generality, the maximum
principle then implies that
\[
v_k(x,t)>0
\]
for all $(x,t)\in\Om_k$. In particular, for each $t\in (k-\frac15,
k+\frac65)$ one can take a point $\widehat X_t^k$ (a global maximum at frozen time~$t$) satisfying
\[
v_k(\widehat X_t^k,t) = \sup_{|x-\ga(t)+(-1)^k \frac{\de_k} 2 e|<\frac{\de_k} 4} v_k(x,t)>0\,.
\]
Notice that this point is not necessarily unique, so if there is more
than one global maximum we will choose
$\widehat X^k_t$ so that the distance of this maximum to the set
\begin{equation*}
S^k_t:=\bigg\{ x\in\RR^n: \Big|x-\ga(t)+(-1)^k \frac{\de_k} 2 e\Big|=\frac{\de_k} 4\bigg\}
\end{equation*}
is as small as possible. If there are more than one maxima that
minimize this distance, any of them will do.

Since these maxima are in the interior of the spacetime
domain, the quantity
\[
\ep_k:=\inf \bigg\{
|\widehat X^k_t-x| : k-\frac15<t< k+\frac65\,,\; \Big|x-\ga(t)+(-1)^k \frac{\de_k} 2 e\Big|=\frac{\de_k} 4\bigg\}
\]
is strictly positive (although possibly not uniformly), and not larger
than~$\de_k/4$. We will set
\[
\eta_k:= \inf \bigg\{
v_k(\widehat X^k_t,t)-v_k(x,t) : k-\frac15<t< k+\frac65\,,\; \Big|x-\ga(t)+(-1)^k \frac{\de_k} 2 e\Big|>\frac{\de_k} 4-\frac{\ep_k}2\bigg\}\,.
\]
Notice that this quantity is positive because we have chosen the
maximum $\widehat X^k_t$ so as to minimize the distance to the set~$S^k_t$.

Let us now define the closed sets
\[
\widetilde\Om_k:=\bigg\{(x,t):  k-\frac18\leq t \leq k+\frac98,\;
\Big|x-\ga(t)+(-1)^k \frac{\de_k} 2 e\Big|\leq \frac{\de_k} 4-\frac{\ep_k}2\bigg\}
\]
and
\begin{equation}\label{defOm}
\Om:=\bigcup_{k=-\infty}^\infty \widetilde\Om_k\,.
\end{equation}
Notice that, by the definition of $\ep_k$ and $\eta_k$,
\[
\inf \bigg\{ v(\widehat X^k_t,t)- v(x,t):  k-\frac18\leq t \leq
k+\frac98\,,\;  \Big|x-\ga(t)+(-1)^k \frac{\de_k} 2 e\Big|=\frac{\de_k} 4-\frac{\ep_k}2\bigg\} \geq{\eta_k}\,.
\]

The function $v$ defined by setting
\[
v(x,t):= v_k(x,t)\qquad \text{if } (x,t)\in \widetilde\Om_k\text{ for
  some }k
\]
then satisfies the equation $Lv=0$ in an open neighborhood of the
closed set~$\Om$. Since $\Om$ is the locally finite union of compact sets such that the
complement of $\Om(t)$ is connected for all $t\in\RR$,
Lemma~\ref{L.technical} ensures that there is a function $u$
satisfying the equation
\[
Lu=0
\]
in $\RR^{n+1}$ and such that
\[
\|u-v\|_{r+2,\widetilde\Om_k}<\frac{\eta_k}3
\]
for all integers~$k$.

It then follows that for all $k-\frac18<t< k+\frac98$ and all $x$ with
\[
\Big|x-\ga(t)+(-1)^k \frac{\de_k} 2 e\Big|=\frac{\de_k} 4-\frac{\ep_k}2
\]
we have
\begin{align*}
u(\widehat X^k_t,t)-u(x,t)&> \Big[v_k(\widehat X^k_t,t)-\frac{\eta_k}3\Big]-
\Big[v_k(x,t)+\frac{\eta_k}3\Big] \\
&= \big[ v_k(\widehat X^k_t,t) - v_k(x,t)\big] -\frac{2\eta_k}3\\
&\geq\frac{\eta_k}3>0\,,
\end{align*}
where we have used the definitions of the various sets and
quantities. It then follows that, for each $t\in (k-\frac18,k+\frac98)$, the function
$u(\cdot, t)$ must have a local maximum $X^k_t$ in the region
\[
\bigg\{x\in\RR^n: \Big|x-\ga(t)+(-1)^k \frac{\de_k} 2 e\Big|< \frac{\de_k} 4-\frac{\ep_k}2\bigg\}\,.
\]
The first assertion of the theorem then follows if one now chooses the
constants $\de_k$ so that
\[
\de_k<\inf_{|t|<|k|+2}\de(t)\,,
\]
where $\de(t)$ is the function appearing in the statement of the theorem.
Notice that this
local maximum is not necessarily unique and that, in general,
$ X^k_t\neq \widehat X^k_t$.

Let us now pass to second part of the theorem, which asserts that if
$L$ is the heat operator and one only controls the existence of a local
hot spot in a finite time interval $[T_1,T_2]\subset [0,\infty)$, then
one can take a global solution $\tu$ on~$\RR^{n+1}_+$ given by a
smooth compactly supported initial datum~$u_0$. The result follows from the above
construction, but instead of defining~$\Om$ as in~\eqref{defOm} one
sets
\[
\Om:=\bigcup_{k=-1}^K\widetilde\Om_k\,,
\]
where $K$ is any integer larger than~$T_2$. The construction is then essentially
as above but instead of using Lemma~\ref{L.technical} to obtain a
better-than-uniform approximation of
the function~$v$, one directly uses Theorem~\ref{T2}. Replacing $t$ by
$t+1$, Theorem~\ref{T2} ensures that one can
approximate~$v$ on the set~$\Om$ (which is now compact) by a solution~$w(x,t)$
to the Cauchy problem
\[
\frac{\pd w}{\pd t}-\De w=0\quad\text{in }\RR^n\times(-1,\infty)\,,\qquad w(x,-1)= f(x)\,,
\]
where $w(x,t)$ tends to zero as $|x|\to\infty$ for all~$t>-1$ and
$f\in C^\infty_c(\RR^n)$. Taking a smooth cutoff function~$\chi_1(\cdot/\ep)$ as in the
proof of Theorem~\ref{T2} and taking the smooth compactly supported function
\[
u_0(x):= \chi_1(x/\ep)\, w(x,0)
\]
for small enough~$\ep$, one infers that $\tu:=e^{t\De}u_0$
approximates~$v$ on~$\Om$ within an error as small as one wishes:
\[
\|\tu-v\|_{C^l (\Om\cap \RR^{n+1}_+)}<\de_0\,.
\]
The theorem then follows.

\begin{remark}
A minor modification of the proof permits to control more than one hot
spot (even countably many) simultaneously. One only needs to consider
the analog of the domain~$\Om$ defined for each curve, and it is not
hard to see that intersections or even accumulation points of the set
of curves do not introduce any serious complications. With a little
work, one then finds the following generalization of Theorem~\ref{T.cp}:

\begin{theorem}\label{T.many}
Let $\ga_i:\RR\to\RR^n$ be a family of (possibly intersecting) parametrized
curves in space labeled
by the integers $i$ in a possibly infinite subset $\cI\subseteq \ZZ$. Take
any positive continuous functions on the line ${\de_i(t)}$. Then
there is a solution to the parabolic equation $Lu=0$ on
$\RR^{n+1}$ such that, for each time $t\in\RR$ and $i\in\cI$,
$u$ has a local hot spot $X_t^i$ with
\[
|X_t^i-\ga_i(t)|<{\de_i(t)}\,.
\]
Furthermore, if $L=-\pd_t+\De$, $\cI$ is finite and $[T_1,T_2]\subset[0,\infty)$ is any
finite interval, for any~$\de_0>0$ there is a function $u_0\in C^\infty_c(\RR^n)$ such that, for each time $t\in [T_1,T_2]$
and $i\in\cI$, $\tu:=e^{t\De}u_0$ has a local hot spot $\tX_t^i$ with
\[
|\tX^i_t-\ga_i(t)|<\de_0\,.
\]
\end{theorem}

\end{remark}

\section{Proof of Theorem~\ref{T.levelsets}}
\label{S.levelsets}

Let $D_i$ denote the bounded domain of $\RR^n$ whose boundary
is~$\Si_i$. We start by constructing, for each integer $i\in\cI$, a
function $v_i(x,t)$ on $D_i\times(i-\frac15,i+\frac65)$ that satisfies
the equation
\[
Lv_i=0
\]
with initial and boundary conditions
\[
v_i\Big(i-\frac15,x\Big)=\phi_i(x)\,,\qquad v_i(t,\cdot)|_{\Si_i}=0\,.
\]
Here $\phi_i$ is any smooth, nonnegative function supported on $D_i$ that is
not identically zero. Since the coefficient $c(x,t)$ is nonpositive,
the maximum principle ensures that
\[
v_i(x,t)>0
\]
for all $(x,t)\in D_i\times(i-\frac15,i+\frac65)$, and Hopf's boundary
point lemma shows that the normal derivative of $v_i$ is negative:
\[
\frac{\pd v_i}{\pd\nu}<0\,.
\]

Let us now fix some positive constants~$\ep_i$, with
$i\in\cI$. Fixing~$t$ and regarding it simply as a
parameter, Thom's isotopy
theorem (see e.g.~\cite[Theorem 20.2]{AR} or~\cite[Theorem~3.1 and Remark~3.2]{Adv}) then ensures that there is some $\al_i>0$
such that, for any $t\in
[i-\frac18,i+\frac98]$, given a positive constant $\al\leq \al_i$ there is a
diffeomorphism~$\Phi_t^i$ of $\RR^n$ with
\[
\|\Phi^i_t-\id\|_{C^{r+2}}<\ep_i
\]
and
such that
\begin{equation}\label{Phi1}
\Ga_{i,\al}(t)=\Phi^i_t(\Si_i)\,,
\end{equation}
where
\[
\Ga_{i,\al}(t)=\{x\in D_i : v_i(x,t)=\al\}\,.
\]
It is worth mentioning that Thom's isotopy theorem does not grant
exactly~\eqref{Phi1}, but only the existence of a connected component
of~$\Ga_{i,\al}(t)$ of the form~$\Phi^i_t(\Si_i)$. However, the
maximum principle ensures that, for $t\geq i-\frac18$, if $\al$ is small
enough (say smaller than $\al_i$, without any loss of generality), the
set~$\Ga_{i,\al}(t)$ is connected. Indeed, $v_i(\cdot,t)$ is
strictly positive on~$D_i$ for $t> i-\frac15$, so for very
small~$\al$ the level set~$\Ga_{i,\al}(t)$ must be contained in a
small (uniformly in~$t\in [i-\frac18,i+\frac98]$) neighborhood
of~$\Si_i$. But Thom's isotopy theorem (which is basically an
application of the implicit function theorem) guarantees that the only
connected component in a small neighborhood of~$\Si_i$
is~$\Phi^i_t(\Si_i)$, which implies~\eqref{Phi1}.

Observe that, for small enough $\al_i$ and $t\in
[i-\frac18,i+\frac98]$, $\Ga_{i,\al}(t)$ is the
boundary of the set
\[
D_{i,\al,t}:=\{x\in D_i : v_i(x,t)>\al\}\,.
\]
Moreover, if $\al_i$ is small enough, $\nabla v_i$ does not vanish on
$\Ga_{i,\al}(t)$, so, again by Thom's isotopy theorem, there is some $\eta_i>0$ such that, if $w$ is a
function with
\[
\max_{t\in[i-\frac18,i+\frac98]} \|v_i(t,\cdot)-w(t,\cdot)\|_{C^{r+2}(D_{i,\frac\al2,t})} <\eta_i\,,
\]
then we have
\begin{equation}\label{ls}
\{x\in D_{i,\frac\al2,t}: w(x,t)=\al\}=\tPhi^i_t(\Si_i)
\end{equation}
for all $t\in[i-\frac18,i+\frac98]$, where $\tPhi^i_t$ is a
diffeomorphism of~$\RR^n$ with
\[
\|\tPhi^i_t-\id\|_{C^{r+2}}<2\ep_i\,.
\]
Notice that,  if $i$ and
$i+1$ are in~$\cI$, as the intersection $D_i\cap D_{i+1}$ is empty, so
is $D_{i,\al,t}\cap D_{i+1,\al',t}$.

We are now ready to complete the proof of
Theorem~\ref{T.levelsets}. Define the closed set
\[
\Om:=\bigcup_{i\in\cI} \Om_i\,,
\]
with
\[
\Om_i:= \bigg\{ (x,t)\in \RR^n\times
\Big[i-\frac18,i+\frac98\Big]: x\in \overline{D_{i,\frac{\al_i}2,t}}\bigg\}
\]
and observe that the previous discussion ensures that the closed sets $\Om_i$
are pairwise disjoint. If we define a function~$v$ on~$\Om$ by setting
\[
v(x,t):=v_i(x,t)
\]
if $t\in [i-\frac18,i+\frac98]$ and $x\in
\overline{D_{i,\frac{\al_i}2,t}}$ for some $i\in\cI$, it is then clear
that $Lv=0$ in $\Om$.

As $\Om$ is a locally finite union of pairwise disjoint sets and the
complement of $\Om(t)$ in $\RR^n$ is connected for all~$t$,
Lemma~\ref{L.technical} then ensures that there is a function satisfying the
equation
\[
Lu=0
\]
in $\RR^{n+1}$ such that
\[
\|u-v\|_{r+2,\Om_i}<\eta_i\,.
\]
Equation~\eqref{ls} then implies that, if $t\in [i-\frac18,i+\frac98]$
for some $i\in\cI$, we have that there is a connected component of the set
\begin{equation}\label{setkk}
\{x\in \RR^n: u(x,t)=\al_i\}
\end{equation}
given by
\[
\Phi^i_t(\Si_i)\,,
\]
where $\Phi^i_t$ is a diffeomorphism of~$\RR^n$ with
$\|\Phi^i_t-\id\|_{C^{r+2}}<2\ep_i$. This is the only
connected component of the set~\eqref{setkk} that intersects with (and is
in fact contained in)~$D_{i,\frac{\al_i}2,t}$. This proves the first part of the
theorem.

Let us now prove the three final assertions of the theorem. Firstly,
notice that if $\cI$ is finite, one can obviously take the same~$\al$
(e.g., $\al:=\min_{i\in\cI}\al_i$) for all~$i$. Concerning the decay,
notice that if $\cI$ is finite, the
set $\Om$ is compact, so when $L$~is the heat operator one can use
Theorem~\ref{T2} just as in the case of Theorem~\ref{T.levelsets} to show that there is a solution $u=e^{t\De} u_0$
with $u_0\in C^\infty_c(\RR^n)$ such that
\[
\|u-v\|_{C^l(\Om)}<\de_0\,.
\]

Finally, suppose that the coefficients of~$L$ are analytic in~$x$. By the density of analytic functions in the space of $C^2$~functions,
without loss of generality one can henceforth assume that the hypersurfaces
$\Si_i$ are analytic too. It is then standard that $v_i(x,t)$ is
analytic in~$x$ up to the boundary, so in particular there is a
slightly larger domain $D_i'\supset \overline{D_i}$ such that the
equation $Lv_i=0$ holds in $D_i'\times (i-\frac15,i+\frac65)$. Thom's
isotopy theorem then ensures that one can take $\al_i=0$, no matter
whether the set $\cI$ is finite or not.

\section{The heat equation on the torus}
\label{S.torus}

Let us first recall Theorems~\ref{T.cp}
and~\ref{T.levelsets} about local hot spots and isothermic hypersurfaces of solutions
to (in particular) the equation $\pd_t v-\De v=0$ on $\RR^{n+1}$ and Remark
\ref{rem:torus}, which ensures the structure stability of these
objets. With these theorems in mind, the results
presented in Theorems~\ref{T.torus} and~\ref{T.torus2} about local hot
spots and isothermic hypersurfaces of solutions to the heat equation on $\TT^n\times\RR$ will stem from the following
lemma. Informally speaking, this lemma ensures that, given a
solution~$v$ of the heat equation
on a certain compact spacetime region
$\Om\subset\RR^{n+1}$, there is a solution~$u$ of the heat equation on
$\TT^n\times\RR$ that approximates~$v$ on~$\Om$ after a suitable
parabolic rescaling of variables in~$\TT^n\times\RR$. In this section, as the coefficients of the heat equation are
constant, for simplicity we will use regular H\"older norms instead of
the parabolic ones.

\begin{lemma}\label{L.torus}
Let $\Om$ be a compact subset of $\RRt_+$ such $\Om^\cc(t)$ is connected for all $t\in\RR$. Given a
function $v$ that satisfies the heat equation $\pd_tv-\De v=0$ in~$\Om$, a positive
real~$k$ and $\de>0$, there
exists an arbitrarily small~$\ep>0$ and a solution $u(x,t)$ of the heat
equation $\pd_t u-\De u=0$ on $\TT^n\times\RR$, whose initial datum is
in fact a trigonometric polynomial, which approximates~$v$ modulo
a rescaling as
\[
\big\|u(\sqrt\ep\; \cdot,\ep \;\cdot)-v\big\|_{C^k(\Om)}<\de\,.
\]
\end{lemma}

\begin{proof}
Theorem~\ref{T2} ensures that there is $f\in C^\infty_c(\RR^n)$ such
that the solution to the heat equation on $\RR^{n+1}_+$ given by $w:=e^{t\De} f$ approximates~$v$ as
\begin{equation}\label{last1}
\|v-w\|_{C^k(\Om)}<\de'
\end{equation}
for any $\de'>0$. In terms of the Fourier transform of~$f$,
\begin{equation*}
w(x,t)=\int_{\RR^n} e^{ix\cdot\xi -t|\xi|^2} \hf(\xi)\, d\xi\,,
\end{equation*}
where the Fourier transform~$\hf$ is a Schwartz function
on~$\RR^n$.



It is well known that any finite
regular complex-valued Borel measure on~$\RR^n$, such as
\[
\hf(\xi)\, d\xi\,,
\]
can be approximated in
the weak topology by a purely discrete
measure, that is, a finite linear combination of the form
\[
\sum_{j=-J}^J a_j\, \de(\xi-\xi_j)
\]
with $\de(\xi-\xi_j)$ the Dirac measure supported on a point~$\xi_j\in
\RR^n$ and $a_j\in\CC$. Since the set of rational
points~$\mathbb Q^n$ is dense in~$\RR^n$, in fact the points~$\xi_j$ can be
taken in~$\mathbb Q^n$. Moreover, we can also take
\[\xi_{-j}=-\xi_j \quad \text{ and } \quad a_{-j}=\overline{a_j}
\]
because $\hf(\xi)=\overline{\hf(-\xi)}$ as~$f$ is real-valued.
Taking a bounded neighborhood $\Om'$ of~$\overline\Om$,
it follows that
\begin{equation}\label{last3}
\|w_1-w\|_{L^\infty(\Om')}<\de'
\end{equation}
where
\[
w_1(x,t):=\sum_{j=-J}^J a_j\, e^{ix\cdot\xi_j -t|\xi_j|^2}
\]
with $a_j\in\CC$, $\xi_j\in \mathbb Q^n$ as above. Notice that $w_1$
is a real-valued function by the choice of the constants $a_j$ and the
points $\xi_j$.

Combining the
inequalities~\eqref{last1}--\eqref{last3} one gets
\[
\|v-w_1\|_{L^\infty(\Om')}<2\de'\,,
\]
and the fact that $\pd_tv-\De v=\pd_tw_1-\De w_1=0$ allows us to use parabolic estimates to promote the above
uniform bound to a $C^k$~estimate of the form
\begin{equation}\label{last4}
\|v-w_1\|_{C^k(\Om)}<C\de'\,.
\end{equation}

Let $M$ denote the least common multiple of the denominators of the
rational points $\xi_j$, so that $M\xi_j\in \ZZ^n$ for all $1\leq
j\leq J$, and let be~$N$ any positive integer. It then follows that the function
\[
u(x,t):= \sum_{j=-J}^J a_j\, e^{iMN x\cdot\xi_j -M^2N^2 t|\xi_j|^2}
\]
is well defined on $\TT^n\times\RR$, as the periodicity in~$x$ follows
from the fact that $NM\xi_j\in\ZZ^n$. An elementary computation also
shows that
\[
\frac{\pd u}{\pd t}-\De u=0\,.
\]
Moreover, setting
\[
\ep:=\frac1{M^2N^2}
\]
it follows directly from~\eqref{last4} and the definition of~$u$ that
\[
\big\|u(\sqrt\ep\; \cdot,\ep \;\cdot)-v\big\|_{C^k(\Om)}<C\de'\,
\]
The claim then follows by choosing~$\de'$ small enough.
\end{proof}

\section*{Acknowledgments}

The authors are supported by the ERC Starting Grants~633152 (A.E.\ and
M.A.G.-F.) and~335079
(D.P.-S.). This work is supported in part by the
ICMAT--Severo Ochoa grant
SEV-2015-0554.

\appendix
\section{Approximation theorems with decay\\ for essentially
  flat elliptic operators}\label{Appendix}

Our goal in this Appendix is to provide a global approximation theorem
for elliptic equations whose coefficients satisfy certain conditions
at infinity. Before passing to our result, let us recall Browder's global approximation theorem
(without any decay conditions) for general elliptic equations, which
can be understood as the elliptic analog of
Theorem~\ref{T1}. Throughout this section, $D^\cc:=\RR^n\backslash D$
will denote the complement of the set~$D$ in~$\RR^n$.

\begin{theorem}[\cite{Browder,Weinstock}]\label{T.Browder}
Let~$D$ be a closed subset in $\RR^n$ such that $D^\cc$ does not have
any bounded connected components and suppose that $P$ is
an elliptic operator with coefficients of class~$C^r$, with $r>1$ a
non-integer real. Given a
function $v$ that satisfies the equation $Pv=0$ on $D$ and $\de>0$, there exists a solution $u$ of
the equation $Pu=0$ in the whole
space $\RR^n$ that approximates~$v$ as
\[
\|u-v\|_{C^{r+2}(D)}<\de\,.
\]
\end{theorem}

We are now ready to state and prove our approximation theorem with
decay, which applies to the class of elliptic operators that we will call {\em essentially flat}\/:

\begin{definition*}
An elliptic operator on $\RR^n$
\[
Pu:=\sum_{i,j=1}^n a_{ij}(x)\frac{\pa^2 u}{\pa x_i\pa x_j}
+\sum_{i=1}^n b_i(x)\frac{\pa u}{\pa x_i}+c(x)u
\]
is {\em essentially flat}\/ if there are constants~$R_0>0$, $c_0>0$ and a function bounded as
\[
|F(r)|\leq C(1+r)^{-1-\ep}
\]
on $(0,\infty)$, with $\ep>0$,
such that
\begin{equation*}
a_{ij}(x)=\de_{ij}\,,\quad b_i(x)=0\,, \qquad c(x)=c_0+F(|x|)
\end{equation*}
for all $|x|>R_0$.
\end{definition*}

The main result of this Appendix can now be stated as follows:

\begin{theorem}\label{T.appendix}
Let $D$ be a compact subset in $\RR^n$ such that $D^\cc$ is connected and suppose that $P$ is
an essentially flat elliptic operator with coefficients of class
$C^r$, with $r>1$ a non-integer real. Given a
function $v$ that satisfies the equation $Pv=0$ on $D$ and $\de>0$, there exists a solution $u$ of
the equation $Pu=0$ in the whole
space $\RR^n$ that approximates~$v$ as
\[
\|u-v\|_{C^{r+2}(D)}<\de
\]
and satisfies the decay condition
\begin{equation}
\sup_{R>0}\frac1R\int_{B_R}|u(x)|^2\, dx<  \infty\,.
\end{equation}
\end{theorem}

\begin{remark}
We recall~\cite{Acta} that if $P$ is the Helmholtz operator ($Pu:=\De u+ c_0u$ with
$c_0>0$), then the function decays pointwise as
\begin{equation}
|u(x,t)|\leq C(1+|x|)^{-\frac{n-1}2}\,.
\end{equation}
\end{remark}

\begin{proof}
Let us take a domain $D'$ containing the closed set~$D$. Without any loss of generality one can assume that the domain~$D'$ is
contained in the ball $B_{R_0}$. Theorem~\ref{T.Browder} ensures that for any $\de>0$ there exists a solution $w$ of the equation $Pw=0$ in the whole space $\RR^n$ such that
\begin{equation*}
\|w-v\|_{C^{r+2}(D')}<\de.
\end{equation*}

Let us now study the behavior of the functions~$w$ in a (one-sided)
neighborhood of~$\pd B_{2R_0}$. We start by expanding $w$ in an orthonormal
basis of spherical harmonics on the unit $(n-1)$-dimensional sphere,
which we denote by
\[
\{Y_{mk}(\om): m\geq 0\,,\; 1\leq k\leq d_m\}
\]
just as in Section~\ref{S.T2}. This leads to the expansion
\[
w(x)= \sum_{m=0}^\infty\sum_{k=1}^{d_m} w_{mk}(r)\, Y_{mk}(\om)\,,
\]
where
\[
w_{mk}(r):=\int_{\SS^{n-1}} w(r\om)\, Y_{mk}(\om)\, d\om\,.
\]
The sum converges in
$H^{r+2}(B_{2R_0})$, so one can take a large enough~$M$ to
ensure that
\begin{equation}\label{AconvM}
\| w-\tw\|_{H^{r+2}(B_{2R_0})} <{\de}\,,
\end{equation}
where $\de$ is a positive constant to be specified later and
\[
\tw:=\sum_{m=0}^{M} \sum_{k=1}^{d_m}w_{mk}(r)\,
Y_{mk}(\om)\,.
\]

Since
\[
Pw = \De w + [c_0+F(|x|)]\,w
\]
for $|x|>R_0$ by the definition of an essentially flat operator, it
automatically follows from the fact that $Y_{mk}(\om)$ are linearly
independent that the functions~$w_{mk}(r)$
satisfy the ODE
\begin{equation}\label{AODE1}
w_{mk}''+\frac{n-1}r w_{mk}' + \bigg( c_0 +F(r)-\frac{\mu_m}{r^2}\bigg) w_{mk}=0
\end{equation}
for $R_0<r<2R_0$, where we recall that
\[
\mu_m:= m(m+n-2)\,.
\]
This is a linear ODE whose coefficients are
bounded for $r\geq R_0$, so it is standard that $w_{mk}$ is indeed
global, that is, it can be extended as a function in
$C^2\loc((0,\infty))$ that satisfies~\eqref{AODE1} for $r>R_0$.

Our goal now is to show that $w_{mk}$ falls off at infinity as
\begin{equation}\label{Aboundvp}
|w_{mk}(r)|<C(1+r)^{-\frac{n-1}2}\,.
\end{equation}
To see this, we rewrite Equation \eqref{AODE1} as
\begin{equation}\label{AODE2}
(r^{\frac{n-1}{2}}w_{mk})''+\left(c_0-\frac{\mu_m+\frac14(n-1)(n-3)}{r^2}+F(r)\right)(r^{\frac{n-1}{2}}w_{mk})=0\,.
\end{equation}
As $c_0$ is positive, the solutions of  $y''+c_0
y=0$ (which are sines and cosines) are obviously bounded,
so
the Dini--Hukuwara theorem~\cite{Bellman} ensures that the solutions
of~\eqref{AODE2} are also bounded as
\[
\|r^{\frac{n-1}{2}}w_{mk}\|_{L^\infty((0,\infty))}<\infty
\]
provided that
\[
\int_{R_0}^\infty \bigg|\frac{\mu_m+\frac14(n-1)(n-3)}{r^2}-F(r)
\bigg|\, dr<\infty\,.
\]
Since $ F(r)<C(1+r)^{-1-\ep}$ by the definition of an essentially flat
elliptic operator, this condition
is always satisfied, so we infer the bound~\eqref{Aboundvp}. In
particular, this shows that the function $\tw$ is defined in all
of~$\RR^n$ and falls off at infinity as
\begin{equation}\label{Adecaytpsi}
|\tw(x)|<C(1+|x|)^{-\frac{n-1}2}\,.
\end{equation}

Let us now observe that, since
\[
P\tw =0
\]
outside $B_{R_0}$, the function
\[
f:= P\tw
\]
is supported in $B_{R_0}$ and bounded as
\[
\|f\|_{H^r(B_{R_0})}<{C\de}
\]
by  the
estimate~\eqref{AconvM}. Suitable resolvent estimates then show that the function
\begin{equation}\label{Agk}
g:=(P+i0)^{-1} f
\end{equation}
satisfies the equation
\[
Pg=f
\]
on $\RR^n$ and the sharp decay condition of
Agmon--H\"ormander:
\begin{equation}\label{AAH}
\bigg(\sup_{R>0}\frac1R\int_{B_R} g^2\, dx\bigg)^{\frac12}\leq C\de\,.
\end{equation}
We refer the reader to~\cite[Theorem 30.2.10]{Hormander} for the
meaning of $(P+i0)^{-1} $ and the associated estimates.

If we now define
\[
u(x):=\tw(x) -g \,,
\]
we infer that it satisfies the equation $Pu=0$ on $\RR^n$, falls off at infinity as
\[
\sup_{R>0}\frac1R\int_{B_R} u^2\, dx<\infty
\]
by the bounds~\eqref{Adecaytpsi} and~\eqref{AAH}, and is close to~$v$ in the sense that
\begin{equation*}
\|u-v\|_{L^2(D')}\leq
\|u-\tw\|_{{L^2 (B_{R_0})}}+\|w-\tw\|_{{L^2(B_{R_0})}} + \|v-w\|_{{L^2(D')}} <C \de
\end{equation*}
by~\eqref{AAH} and the definitions of $\tw$ and $w$. Furthermore, since $P(u-v)=0$ in~$D'$, standard
elliptic estimates yield the H\"older bound
\[
\|u-v\|_{C^{r+2}(D)}<C\de\,.
\]
The theorem then follows.
\end{proof}

\bibliographystyle{amsplain}

\end{document}